\documentclass[11pt]{article}
\usepackage{url}
\usepackage{verbatim}
\usepackage[titletoc]{appendix}
\usepackage{graphicx}
\usepackage{stmaryrd}
\usepackage{amsmath}
\usepackage{amsthm}
\usepackage{amsfonts}
\usepackage{graphicx}
\usepackage{nicefrac}
\usepackage{longtable}
\usepackage{color}
\usepackage{graphicx, amssymb,graphics}
\usepackage{epstopdf}
\usepackage{float}
\usepackage{bm}
\usepackage{enumitem}

\usepackage{multirow,booktabs}
\usepackage{algorithmic}
\usepackage{subfigure}

\allowdisplaybreaks
\renewcommand \d {\mathrm{d}}

\newtheorem{theorem}{Theorem}[section]

\newtheorem{lemma}{Lemma}[section]

\def\b[#1]{\boldsymbol{#1}}

\newcommand\be {\begin{equation}}
\newcommand\ee {\end{equation}}

\title{Bound preserving and mass conservative methods for the nonlocal Cahn-Hilliard equation with the logarithmic Flory-Huggins potential}

\author{
Yingying Wang\thanks{School of Mathematical Sciences, Beijing Normal University, Beijing 100875, China. Email: yywang@mail.bnu.edu.cn}
\and
Xiao Li\thanks{Key Laboratory of Mathematics and Complex Systems, Ministry of Education and School of Mathematical Sciences, Beijing Normal University, Beijing 100875, China. Email: lixiao@bnu.edu.cn}
\and
Zhengru Zhang\thanks{Key Laboratory of Mathematics and Complex Systems, Ministry of Education and School of Mathematical Sciences, Beijing Normal University, Beijing 100875, China. Email: zrzhang@bnu.edu.cn}
}

\date{}

\begin{document}

\maketitle
\numberwithin{equation}{section}

\begin{abstract}
It is well known that the exponential time differencing (ETD) method has been successfully applied to the classic Cahn-Hilliard equation with double well potential. However, this numerical method can not be extended to the Cahn-Hilliard equation with Flory-Huggins potential directly due to the fact that the the numerical solution may go beyond the physical interval which leads the non-physical solution. 
In this paper, we develop and analyze first- and second-order numerical schemes for the nonlocal Cahn-Hilliard equation with the classic Flory-Huggins energy potential.
In more detail, the ETD method is firstly used to obtain the prediction solution, and then this prediction solution is corrected by the projection method to avoid non-physical solution. The proposed method is shown to preserve bound and mass conservation in discrete settings. In addition, error estimates for the numerical solution are rigorously obtained for both schemes. Extensive numerical tests and comparisons are conducted to demonstrate the performance of the proposed schemes.
\end{abstract}

\medskip
\section{Introduction}
Let us consider the nonlocal Cahn-Hilliard (NCH) equation
\begin{subequations} \label{nonlocal_eq}
\begin{align}
u_t & = \Delta \mu\\
\mu & =\frac{\theta}{2} \ln(1+u)-\frac{\theta}{2} \ln(1-u) - \theta_c u - \epsilon^2 \Delta u + \sigma (-\Delta)^{-1}(u-\overline{u} ) , 
\end{align}
\end{subequations}
for $t>0$ and $\bm{x}\in\Omega$. Here, $u=u(\bm{x},t)$ is the unknown function, and $\Omega\subset\mathbb{R}^d$ is the spatial domain. 
The parameter $\epsilon>0$ is the capillary coefficient, which is related to the interfacial thickness. The constants $\theta$ and $\theta_c$, which are associated with the mixture temperature, satisfy $0<\theta<\theta_c$.
The parameter $\sigma>0$ denotes the strength of the nonlocal interaction. We also impose the initial data
$u(\bm{x},0)=u^0 $
and the periodic boundary condition. 

The Eq.\eqref{nonlocal_eq} usually can be regarded as the $H^{-1}$ gradient flow with respect to the energy functional with logarithmic potential:
\begin{equation}\label{energy_log}
E(u)=\int_{\Omega}\left(\frac{\theta}{2} [(1+u)\ln (1+u)+(1-u)\ln(1-u)]-\frac{\theta_c}{2}u^2+\frac{\epsilon^2}{2} |\nabla u|^2 +\frac{\sigma}{2} |(-\Delta)^{-\frac{1}{2}}(u-\overline{u})|^2 \right) \d \bm{x},
\end{equation}
Due to the gradient structure of \eqref{nonlocal_eq}, the following energy dissipation laws formally hold:
\begin{equation*}
\frac{\d}{\d t} E(u(t))=-\int_{\Omega}|\nabla \mu |^2\d\bm{x}\leq 0. 
\end{equation*}
The theoretical solution are expected to remain in the interval $(-1,1)$ in a point-wise sense and preserve the mass conservation, i.e., $\int_{\Omega} u \,\d\bm{x}=\int_{\Omega} u^0 \,\d\bm{x}$ for all $t>0$. Especially, the major challenge lies in the preservation of the bound preserving property. 
To circumvent this issue, the logarithmic nonlinear term is often approximated by regularized terms. The resulting chemical potential corresponds to the classical diblock copolymer model originally proposed by Ohta and Kawasaki \cite{OK1986}. There are many excellent works for Ohta-Kawasaki model. Nishiura and Ohnishi \cite{framework} established the mathematical framework.  The implicit midpoint spectral method \cite{linear2014} for the equilibrium and the invariant energy quadratization (IEQ) method \cite{IEQ_intro} were adopted to study this diblock copolymer model; see more related references \cite{IEQ_intro,fifth,Xu2020,dcds,ETD_ref,nonlocal_Cahn,nCH}. 
However, the exact solution may still go beyond the physically meaningful interval $(-1,1)$  with the regularized nonlinear term \cite{CH_PDE}. Therefore, the free energy with the logarithmic potential is often considered to be more physically.

We were drawn to the NCH equation \eqref{nonlocal_eq} not only because of its physical significance, particularly in the self-assembly of diblock copolymers, 
but also because of its broad applications across many scientific fields, including physics, materials science, and financial mathematics. 
For example, \cite{2024opt} developed an optimal control framework for a nonlocal tumor model, which can be used to optimize chemotherapy doses and can also be applied to cell membrane electroporation to improve related device parameters \cite{2023pha}.
In image processing, the model can be used for image inpainting \cite{2018CH}. 
In finance, nonlocal phase field model can be applied to solve local risk minimization pricing and hedging problems under semi-Markov regime switching mechanism \cite{2016sys}.

There are many excellent works on both theoretical analysis and numerical method for NCH equation \eqref{nonlocal_eq}. The existence of weak solutions and their uniqueness, and the existence of the connected global attractor were proven in \cite{existence} and \cite{uniqueness} (see also \cite{both}). Moreover, well-posedness and regularity of weak solutions were studied in \cite{separation}, namely, in this work the authors established the validity of the strict separation property in two dimensions for NCH equation \eqref{nonlocal_eq}. This means that if the initial state is  not a pure state (i.e. $u\equiv 1$ or $u\equiv -1$ ), then the corresponding solution stays away from the pure states in finite time, uniformly with respect to the initial data.  In \cite{3D_separation}, the authors established the strict separation property in three dimensions. 

On the numerical discretization for the NCH equation \eqref{nonlocal_eq}, the bound preserving property is very challenging. The traditional approach to enforce $u\in (-1,1)$ is via regularization of the
nonlinearity or using the technique of sub-differential operators. For example, the second order accurate, unconditionally uniquely, solvable and unconditionally energy stable convex splitting schemes were proposed \cite{convex1_wang} for NCH equation by regularizing  the logarithmic potential. The first and second order linear and unconditional energy stable schemes based on the IEQ method were developed to the Cahn-Hilliard equation \cite{IEQ_flory}. In this article \cite{IEQ_flory}, the author regularized the logarithmic potential by a $\mathcal{C}^2$ piecewise function. Besides, Copetti and Elliott \cite{CH_element}
considered a fully implicit Euler scheme  to the Cahn-Hilliard equation with the finite element
approximation in space. It was shown that under the condition that the time step $\tau$  is sufficiently small, then there exists a unique
numerical solution for the implicit Euler discretization.
Wang et all \cite{convex2_wang} designed the first-order scheme based on the convex-concave decomposition for Cahn-Hilliard equation with the Flory-Huggins functional. By using a variational idea taking
advantage of the singular nature of the nonlinearity, the numerical solution constructed in \cite{convex2_wang} can
be guaranteed to lie in the interval $(-1,1)$ in each iteration. But this numerical scheme based on the convex splitting method is the nonlinear scheme which leads to increased computational cost.

Recently, Shen et al. \cite{lagrange1,lagrange2} constructed some positivity/ bound-preserving and mass conservative  schemes for semilinear and quasi-linear  parabolic equations by adopting the Lagrange multiplier approach and predictor-corrector method, where  the generic semi-implicit scheme is applied in the prediction step, and then at the correction step structure-preserving is forced via Lagrange multipliers with negligible computational cost.  This method has already been successfully applied to many models, such as Poisson-Nernst-Planck
(PNP) equation \cite{PNP}, Keller-Segel equation \cite{KS},  nonlinear convection diffusion equation \cite{convection}. It is worth mentioning that the cut-off approach \cite{cut_off} can be interpreted as a special case of this prediction-correction method. It provides it provides a different justification for the cut-off approach, moreover modifies the cut-off approach so that it becomes mass conserving scheme. 

It is well known that the exponential time differencing (ETD)-based numerical approach is widely used to solve nonlinear parabolic PDEs, in which an exact integration of the linear and positive definite part of the PDE is used, combined with certain explicit approximations to the temporal integral of the nonlinear and concave terms. The ETD scheme was systematically studied in \cite{first_ETD} and then further developed by Cox and Matthews with applications on stiff system \cite{Cox}, where higher-order multistep and Runge-Kutta versions of these schemes were described. Hochbruck and Ostermann provided several nice reviews of the ETD Runge-Kutta methods \cite{RK} and ETD multistep methods \cite{multiple}. ETD schemes usually perform as efficiently as an explicit scheme since the operator exponentials can be often implemented by fast algorithms in regular domains, which leads to many successful applications to simulations of coarsening dynamics, see, e.g. \cite{MBP_ETD,epitaxial,PFC,Rev}. 
However, to the best of our knowledge, there is few work on the ETD schemes for the NCH equation with logarithmic potential, due to the fact that the numerical solution may go beyond the interval $(-1,1)$. 

The main goal of paper is to develop first- and second-order bound preserving and mass conservative schemes for NCH equation \eqref{nonlocal_eq} by using prediction-correction approach and ETD schemes. Specially, at the prediction step, we obtain the prediction solution at each time step by using the ETD schemes which may go beyond the bound, and then, we project the prediction solution into the bound preserving and mass conservation space at the correction step which is equivalent to solve a constrained minimization problem. This constrained minimization problem is transformed into the system with  the Lagrange multiplier and KKT conditions which can be computed by a suitable iteration procedure. To the best of our knowledge, this is the first exploration in the direction of designing the ETD-based method for NCH equation with singular potential. 

The rest of this paper is organized as follows. Section \ref{space} is devoted to spatial discretization of the Eq. \eqref{nonlocal_eq} and some preliminary lemmas to be used in latter sections. Then based on Lagrange multiplier approach, the first- and second-order bound preserving and mass conservative ETD schemes, denoted as P-ETD schemes are developed in Section \ref{schemes}, together with some properties of matrix functions. The fully discrete convergence analysis for the P-ETD schemes are presented in Section \ref{convergence}. In Section \ref{experiments}, extensive numerical tests and comparisons are carried out to illustrate the accuracy and effectiveness of the proposed schemes. Some concluding remarks are given in Section \ref{conclusion}.

\section{Spatial discretization}\label{space}
For simplicity, throughout this paper, we consider the two-dimensional square domain $\Omega=(0,L)\times (0,L)$ for Eq. \eqref{nonlocal_eq} with periodic boundary conditions. Note that extending the analysis to three-dimensional problems is straightforward. In this section, we first introduce some notations related to the spatial discretization by central finite difference.

Given a positive integer $M$, we set $h=L/M$ to be the size of the uniform mesh partitioning $\overline{\Omega}$. Denote by $\Omega_h$ the set of mesh points $(x_i,y_j)=(ih,jh)$, $1\leq i,j\leq M$. For a grid function $v$ defined on $\Omega_h$, we write $v_{ij}=v(x_i,y_j)$ for simplicity. Let $\mathcal{M}_h$ be the set of all periodic grid function on $\Omega_h$, i.e.,
\begin{equation*}
\mathcal{M}_h=\{v:\Omega_h\rightarrow \mathbb{R} |v_{i+kM,j+lM}=v_{ij}, k,l\in \mathbb{Z} ,1\leq i,j\leq M\}.
\end{equation*}
The discrete $L^2$ inner product $\langle\cdot ,\cdot \rangle$, discrete $L^2$ norm $\left\lVert \cdot \right\rVert $ and discrete $L^\infty$ norm $\left\lVert\cdot \right\rVert _\infty$ are defined as follows: for any $v,w\in \mathcal{M}_h$,
\begin{equation*}
\langle v,w\rangle=h^2\sum_{i,j=1}^{M}v_{ij}w_{ij},\quad \left\lVert v\right\rVert=\sqrt{\langle v,v\rangle}, \quad \left\lVert v \right\rVert  _\infty =\max_{1\leq i,j\leq M}|v_{ij}| ,
\end{equation*}
and 
\begin{equation*}
\langle \bm{v},\bm{w}\rangle=\langle v^1,w^1\rangle+\langle v^2,w^2 \rangle,\quad \left\lVert \bm{v}\right\rVert=\sqrt{\langle \bm{v},\bm{v}\rangle} 
\end{equation*}
for any $\bm{v}=(v^1,v^2)^T$,   $\bm{w}=(w^1,w^2)^T\in \mathcal{M}_h\times \mathcal{M}_h$. We apply the second-order central finite difference to approximate spatial differentiation operators. For any $v\in \mathcal{M}_h$, the discrete Laplace operator $\Delta_h$ is defined by 
\begin{equation*}
\Delta_h v_{ij}=\frac{1}{h^2}(v_{i+1,j}+v_{i-1,j}+v_{i,j-1}+v_{i,j+1}-4v_{ij}),\quad 1\leq i,j\leq M,
\end{equation*}
and the discrete gradient operator $\nabla _h$ is defined by 
\begin{equation*}
\nabla_h v_{ij}=\left(\frac{v_{i+1,j}-v_{ij}}{h},\frac{v_{i,j+1}-v_{ij}}{h}\right)^T, \quad 1\leq i,j\leq M.
\end{equation*}
By periodic boundary conditions, the discrete summation-by-parts formula is easy to verify:
\begin{equation*}
\langle v,\Delta_h w\rangle =-\langle \nabla _h v, \nabla _h w\rangle =\langle \Delta_h v,w\rangle, \quad \forall v,w\in \mathcal{M}_h.
\end{equation*}
Obviously, $\Delta_h$ is self-adjoint and negative semi-definite. 
For any function $\varphi :\overline{\Omega}\to \mathbb{R}$, we denote by $I^h$ the operator projecting $\varphi$ on the mesh as $(I^h\varphi)_{ij}=\varphi(x_i,y_j)$ for $1\leq i,j\leq M$. For simplicity, we may drop the operator $I^h$ when there is no ambiguity.

Since $\mathcal{M}_h$ is a finite-dimensional linear space, any grid function $v\in \mathcal{M}_h$ and linear operator $\mathcal{Q} :\mathcal{M}_h\rightarrow \mathcal{M}_h$ can be regarded as a vector in $\mathbb{R} ^{M^2}$ and a matrix in $\mathbb{R} ^{M^2\times M^2 }$, respectively. We still adopt the notations $\left\lVert \cdot\right\rVert $ and $\left\lVert \cdot\right\rVert _\infty$ to represent the matrix-induced norms consistent with $\left\lVert \cdot\right\rVert $ and $\left\lVert \cdot \right\rVert _\infty$ defined for vectors before, respectively.   

\section{Bound preserving and mass conservative schemes}\label{schemes}

In this section, we aim to construct  bound preserving and mass conservation schemes for NCH equation \eqref{nonlocal_eq} based on the prediction-correction approach. The prediction solution is obtained by the classic ETD schemes but it may lose the bound preservation and mass conservation property. Then, at the correction step, the prediction solution is projected into the set which keep bound and mass conservation.

First, let us introduce the space-discrete version of \eqref{nonlocal_eq}. The space-discrete problem is to find a function $u_h(t)\in {\mathcal{M}_h}$ satisfying 

\begin{equation}\label{space-discrete}
\frac{\mathrm{d} u_h}{\mathrm{d} t}=-\epsilon^2\Delta^2_h u_h-\sigma u_h +\frac{\theta}{2} \Delta_h [\ln(1+u_h)-\ln(1-u_h)] -\theta_c \Delta_h u_h+\sigma \overline{u_h}.
\end{equation}
For the sake of the stability of the time-stepping schemes developed later, we introduce a stabilizing
term $\kappa \Delta_h$ with $\kappa > 0$ and  the system \eqref{space-discrete} can be rewritten as
\begin{equation}\label{space_discrete_L_h}
\frac{\mathrm{d}u_h}{\mathrm{d}t}+L_h u_h =F(u_h),
\end{equation} 
where
\begin{equation}
\begin{aligned}
&L_h:=\epsilon^2\Delta_h^2 -\kappa\Delta_h +\sigma I\\
&F(u_h)=\frac{\theta}{2}\Delta_h[\ln(1+u_h)-\ln(1-u_h)]-\theta_c \Delta_h u_h -\kappa \Delta_h u_h+\sigma \overline{u_h}.
\end{aligned}
\end{equation}

Let us partition the time interval into $\{t_n=n\tau\}_{n\geq0}$ with $\tau>0$ being a uniform time step size. In the remaining part of the paper, we will study time integration schemes for the space-discrete system \eqref{space_discrete_L_h}. For simplicity of representation, we denote by $u^n$ the fully discrete approximate value of $u(t_n)$ or $u_h(t_n)$ with $u$ and $u_h$ denoting the exact solutions to the original continuous problem \eqref{nonlocal_eq} and the space-discrete problem \eqref{space_discrete_L_h}, respectively.


\subsection{First-order ETD scheme with the projection method}

Using the  variation-of-constants formula to  \eqref{space_discrete_L_h} and 
setting $t=t_n$ give us 
\begin{equation}\label{exact_solution_t}
u_h(t_{n+1})=\mathrm{e}^{-\tau L_h}u_h(t_n)+ \int_{0}^{\tau} \mathrm{e}^{-(\tau-s)L_h} F(u_h(t_n+s)) \,\d s. 
\end{equation}
Denote by $u^n$ the numerical approximation of $u_h(t_n)$. Approximating $u_h(t_n+s)$ by $u^n$ and calculating the resulting integral exactly, we generate the prediction solution $\widetilde{u}^{n+1}$ by  the first-order ETD (ETD1) scheme defined as the following form: 
\begin{equation}
\widetilde{u}^{n+1}=\mathrm{e}^{-\tau L_h}u^n+\int_{0}^{\tau} \mathrm{e}^{-(\tau-s)L_h}  F(u^n) \,\d s,
\end{equation}
or equivalently,
\begin{equation}\label{ETD1}
\widetilde{u}^{n+1}=\varphi_0(\tau L_h)+\tau \varphi_1(\tau L_h) F(u^n),
\end{equation}
where 
\begin{equation*}
\varphi_0(a)=\mathrm{e}^{-a}, \quad \varphi_1(a)=\frac{1-\mathrm{e}^{-a}}{a}.
\end{equation*}
Next we introduce the bound preserving and mass conservative projection method for the correction step.
For the  prediction solution $\widetilde{u}^{n+1}\in \mathcal{M}_h$, we aim to project it into the admissible  set 
\begin{equation}
{\mathcal{M}_\delta} =\{v\in {\mathcal{M}_h}| \left\lVert v\right\rVert _\infty \leq 1-\delta, \quad \langle v,1\rangle=\langle \widetilde{u}^{n+1},1\rangle \}.
\end{equation}
Then we adopt the discrete $L^2$ projection to satisfy
\begin{equation}\label{arg}
u^{n+1}=\arg \min \limits_{\phi\in {\mathcal{M}_\delta}}\frac{1}{2} \left\lVert \phi-\widetilde{u}^{n+1}\right\rVert ^2.
\end{equation}
This is a convex minimization problem and can be solved efficiently.
By introducing the Lagrange multipliers $\lambda_h^{n+1}$ and $\xi_h^{n+1}$, we obtain the following  Karush-Kuhn-Tucker (KKT) conditions:
\begin{subequations}\label{corrector}
\begin{align}
&u^{n+1}=\widetilde{u}^{n+1}+\lambda_h^{n+1}g'(u^{n+1})+\xi_h^{n+1},\\
&\lambda_h^{n+1}g(u^{n+1})=0,\quad \lambda_h^{n+1}\geq 0, \quad g(u^{n+1})\geq0,\label{1_4_b}\\
&\langle u^{n+1},1\rangle=\langle \widetilde{u}^{n+1},1\rangle \label{mass_conservation},
\end{align}
\end{subequations}
where $g(u)=(1-\delta)^2-u^2$. $\lambda_h^{n+1}\in \mathcal{M}_h$ is the Lagrange multiplier to preserve bound, and $\xi_h^{n+1}\in\mathbb{R} $ is another Lagrange multiplier to enforce mass conservation.
Note that once $\xi_h^{n+1}$ is known, $u^{n+1}=\{u^{n+1}_{i,j}\}$ and $\lambda_h^{n+1}=\{\lambda_{i,j}^{n+1}\}$ can be solved directly:
\begin{equation}\label{10}
[u^{n+1}_{i,j},\lambda_{i,j}^{n+1}]=\left\{
\begin{aligned}
&[\widetilde{u}_{i,j}^{n+1}+\xi_h^{n+1},0],\quad |\widetilde{u}_{i,j}^{n+1}+\xi_h^{n+1} |\leq  1-\delta,   \\
&[-1+\delta,\frac{-1+\delta -(\widetilde{u}_{i,j}^{n+1}+\xi_h^{n+1})}{ g'(-1+\delta)} ], \quad \widetilde{u}_{i,j}^{n+1}+\xi_h^{n+1} < -1+\delta, \\
&[1-\delta, \frac{1-\delta -(\widetilde{u}_{i,j}^{n+1}+\xi_h^{n+1} )}{ g'(1-\delta)} ], \quad \widetilde{u}_{i,j}^{n+1}+\xi_h^{n+1} > 1-\delta. 
\end{aligned}
\right.
\end{equation}
According to \eqref{10}, the discrete mass conservation property \eqref{mass_conservation} can be rewritten as 
\begin{equation}
\sum_{|\widetilde{u}^{n+1}_{ij}+\xi_h^{n+1}|\leq 1-\delta} \left (\widetilde{u}^{n+1}_{ij}+\xi_h^{n+1}\right )h^2+\sum_{\widetilde{u}^{n+1}_{ij}+\xi_h^{n+1}>1-\delta} (1-\delta)h^2 +\sum_{\widetilde{u}^{n+1}_{ij}+\xi_h^{n+1} <-1+\delta}(-1+\delta) h^2=\langle \widetilde{u}^{n+1},1\rangle.
\end{equation}
Let us set 
\begin{equation}
\begin{aligned}
F_n(\xi):&=\sum_{|\widetilde{u}^{n+1}_{ij}+\xi|\leq 1-\delta} \left (\widetilde{u}^{n+1}_{ij}+\xi^{n+1}\right )h^2+\sum_{\widetilde{u}^{n+1}_{ij}+\xi>1-\delta}(1-\delta)h^2 +\sum_{\widetilde{u}^{n+1}_{ij}+\xi <-1+\delta}(-1+\delta)h^2\\
&\quad -\langle \widetilde{u}^{n+1},1\rangle.
\end{aligned}
\end{equation}
 Then $\xi_h^{n+1}$ is a solution to this nonlinear algebraic equation $F_n(\xi)=0$, which can be solved by secant method:
\begin{equation}
\xi_0=0,\quad \xi_1= \tau,\quad 
\xi_{k+1}=\xi_k-\frac{F_n(\xi_k)(\xi_k-\xi_{k-1})}{F_n(\xi_k)-F_n(\xi_{k-1})} \quad k\geq 1.
\end{equation}
Now, we complete the ETD1 projection (P-ETD1) scheme. 
Next we pay attention to  the mass conservation property for the P-ETD1 scheme.
\begin{lemma}\label{first_order_mass_conservation}
The P-ETD1 scheme \eqref{ETD1} and \eqref{corrector} conserves the mass unconditionally, i.e., for any time step size $\tau>0$, the P-ETD1 scheme satisfies 
\begin{equation}
\langle u^{n+1},1\rangle =\langle u^n,1\rangle.
\end{equation}
\end{lemma}

\begin{proof}
In fact, for the P-ETD1 scheme, the prediction step can be expressed as the differential form: to find $W_1(s):[0,\tau]\to \mathcal{M}_h$ satisfying: 
\begin{equation}\label{ODE_ETD1}
\left\{
\begin{aligned}
&\frac{\mathrm{d} W_1(s)}{\mathrm{d}s}+L_h W_1(s)=F(u^n), \quad s\in(0,\tau],\\
&W_1(0)=u^n.
\end{aligned}
\right.
\end{equation}
and then set $\widetilde{u}^{n+1}=W_1(\tau)$.
Taking the  discrete $L^2$ inner product with $1$ on both sides of \eqref{ODE_ETD1}, we obtain 
\begin{equation}
\frac{\mathrm{d}}{\mathrm{d}s} \langle W_1(s),1\rangle +\sigma \langle W_1(s),1\rangle =\sigma \langle u^n,1\rangle.
\end{equation} 
Let $V(s)=\langle W_1(s),1\rangle$, then we have 
\begin{equation}
\frac{\mathrm{d}V(s)}{\mathrm{d}s}+\sigma V(s)=\sigma \langle u^n,1\rangle,
\end{equation}
with initial data $V(0)=\langle u^n,1\rangle $. Multiplying by the exponential term $\mathrm{e}^{s}$ and integrating on the interval $[0,\tau]$, we immediately get 
\begin{equation}
V(\tau)\mathrm{e}^{\tau}-\langle u^n,1\rangle =\langle u^n,1\rangle (\mathrm{e}^{\tau}-1),
\end{equation}
which implies $\langle \widetilde{u}^{n+1},1\rangle=V(\tau) =\langle u^n,1\rangle$. Then combining with discrete mass conservation property \eqref{mass_conservation}, we complete the proof.
\end{proof}

\subsection{Second-order ETDRK2 scheme with the projection method}
The construction of the second-order projection ETDRK2 ( P-ETDRK2) scheme is similar to that of the P-ETD1 scheme: we first obtain a predicted solution by the ETDRK2 scheme, and then apply a projection step to obtain a corrected solution that preserves bound and mass conservation. The specific form is given as follows:

First, we use the P-ETD1 scheme to obtain the middle solution $u^{n+1}_{\text{mid}}$, i.e.,
\begin{subequations} \label{mid_solution}
\begin{align}
&\widetilde{u}^{n+1}_{\text{mid}}=\varphi_0(\tau L_h)+\tau \varphi_1(\tau L_h) F(u^n),\\
&u^{n+1}_{\text{mid}}=\widetilde{u}^{n+1}_{\text{mid}}+\lambda_{h,1}^{n+1}g'(u_{\text{mid}}^{n+1})+\xi_{h,1}^{n+1},\\
&\lambda_{h,1}^{n+1}g(u_{\text{mid}}^{n+1})=0,\quad \lambda_{h,1}^{n+1}\geq 0, \quad g(u_{\text{mid}}^{n+1})\geq0,\\
&\langle u_{\text{mid}}^{n+1},1\rangle=\langle \widetilde{u}_{\text{mid}}^{n+1},1\rangle,
\end{align}
\end{subequations}
and then approximate $F(u_h(t_n+s))$ in \eqref{exact_solution_t} by a linear interpolation between $F(u^n)$ and $F(u^{n+1}_{\text{mid}})$ to obtain the predicted solution, followed by a projection step, yielding
\begin{subequations}\label{P-ETDRK2}
\begin{align}
&\begin{aligned}
\widetilde{u}^{n+1}&=\mathrm{e}^{-\tau L_h}u^n+\int_{0}^{\tau}  \mathrm{e}^{-(\tau-s)L_h}\left((1-\frac{s}{\tau} )F(u^n)+\frac{s}{\tau} F(u^{n+1}_{\text{mid}})\right)\,\d s \\
&=\varphi_0(\tau L_h)u^n+\tau [(\varphi_1(\tau L_h)-\varphi_2(\tau L_h))F(u^n)+\varphi_2(\tau L_h) F(u^{n+1}_{\text{mid}})]\\
\end{aligned}
\\
&\begin{aligned}
&u^{n+1}=\widetilde{u}^{n+1}+\lambda_{h,2}^{n+1} g'(u^{n+1})+\xi_{h,2}^{n+1}
\end{aligned}
\\
&\begin{aligned}
\lambda_{h,2}^{n+1}g(u^{n+1})=0,\quad \lambda_{h,2}^{n+1}\geq 0, \quad g(u^{n+1})\geq0,\\
\end{aligned}
\\
&\begin{aligned}
\langle u^{n+1},1\rangle=\langle \widetilde{u}^{n+1},1\rangle,  \label{mass_conservation2}
\end{aligned}
\end{align}
\end{subequations} 
where 
\begin{equation}
\varphi_2(a)=\frac{\mathrm{e}^{-a}-1+a}{a^2}.
\end{equation}
We know that $\varphi_0(a)$, $\varphi_1(a)$ and $\varphi_2(a)$ are positive when $a\geq0$, where the values at $a=0$  can be
defined in the sense of limit.

The mass conservation property can also be established for the P-ETDRK2 scheme.

\begin{lemma}\label{second_order_mass_conservation}
The P-ETDRK2 scheme \eqref{mid_solution} and \eqref{P-ETDRK2} conserves the mass unconditionally, i.e., for any time step size $\tau>0$, the P-ETDRK2 scheme satisfies
\begin{equation}
\langle u^{n+1},1\rangle =\langle u^n,1\rangle.
\end{equation}
\end{lemma}

\begin{proof}
Similar to P-ETD1 scheme, the prediction step can also be established for P-ETDRK2 scheme as the differential form: to find $W_2(s):[0,\tau]\to \mathcal{M}_h$ satisfying:
\begin{equation}\label{ODE_ETDRK2}
\left\{
\begin{aligned}
&\frac{\mathrm{d} W_2(s)}{\mathrm{d}s}+L_h W_2(s)=(1-\frac{s}{\tau} )F(u^n)+\frac{s}{\tau} F(u^{n+1}_{\text{mid}}) , \quad s\in(0,\tau],\\
&W_2(0)=u^n,
\end{aligned}
\right.
\end{equation}
and then set $W_2(\tau)=\widetilde{u}^{n+1}$.
Taking the discrete $L^2$ inner product with \eqref{ODE_ETDRK2} by $1$ , we have 
\begin{equation}
\frac{\mathrm{d}}{\mathrm{d}s} \langle W_2(s),1\rangle +\sigma \langle W_2(s),1\rangle =\langle u^n,1\rangle,
\end{equation} 
Similar to the proof of Lemma \ref{first_order_mass_conservation}, we obtain $\langle \widetilde{u}^{n+1},1\rangle =\langle u^n,1\rangle$. Then combining with \eqref{mass_conservation2}, we have $\langle u^{n+1},1\rangle=\langle u^n,1\rangle$.
\end{proof}

\section{ Convergence analysis
}\label{convergence}

Suppose that exact solution $u$ is smooth sufficiently, and the following separation property holds  for the exact solution $u$:
\begin{equation}
\left\lVert u\right\rVert _\infty \leq 1-2\delta  \quad \text{for} \quad \delta>0, \quad \text{at a point-wise level}.
\end{equation} 

First we give the following properties for the exponential-type functions  which are crucial to the latter convergence analysis.
\begin{lemma}[see \cite{High2008}]\label{property1}
  Let  $\varphi$ be defined on the spectrum of a real and symmetric matrix $A$ of order $m$, that is, the values $\varphi(\lambda_i)$, $1\leq i \leq m$ exist, where $\{\lambda_i\}_{i=1}^m$ are the eigenvalues of $A$. Then
  \begin{enumerate}[itemsep=0pt,partopsep=0pt,parsep=\parskip,topsep=0pt]
  \item[\rm (1)] $\varphi(A)$ is symmetric and commutes with $A$;
  \item[\rm (2)] the eigenvalues of $\varphi(A)$ are $\{\varphi(\lambda_i):1 \leq i \leq m\}$;
  \item[\rm (3)] $\varphi(P^{-1}AP)=P^{-1}\varphi(A)P$ for any nonsingular matrix $P$ of order $m$;
  \item[\rm (4)] $\frac{\d}{\d s}(\mathrm{e}^{sA})=A\mathrm{e}^{sA}=\mathrm{e}^{sA}A$ for any $s\in \mathbb{R}$.
  \end{enumerate}
\end{lemma}

\begin{lemma}\label{property2}
\begin{enumerate}
\item [\rm (1)] For any $a>0$, the following inequalities hold:
\begin{equation}
\begin{aligned}
&0<(1+a)\varphi_0(a)<1,\\
&1<(1+a)\varphi_1(a)<2,\\
&\frac{1}{2}<(1+a)\varphi_2(a)<1,\\
&0<(1+a)[\varphi_1(a)-\varphi_2(a)]<1. 
\end{aligned}
\end{equation}
\item [\rm (2)] If $0< s< \tau \leq 1$, then for any $a>0$, it holds that 
\begin{equation}
0< (1+a\tau)\mathrm{e}^{-a(\tau-s)}<1.
\end{equation}  
\end{enumerate}
\end{lemma}

\subsection{Fourier projection of the exact solution}

For the numerical analysis, we introduce the Fourier projection for the exaction solution $u$, due to the fact that $I^h u$ is not mass conservative at the discrete level in general, i.e. $\langle I^h u(\bm{x},t_n),1\rangle \neq \langle I^h u(\bm{x},t_{n-1}),1\rangle$.

Define $u_N(\cdot,t):=\mathcal{P}_Nu(\cdot,t)$, the spatial Fourier projection of the exact solution into $\mathcal{B}^N$, the space of trigonometric polynomials of degree up to $N$. The following projection approximation is standard: if $u\in L^\infty((0,T);H_{\text{per}}^l)$, for  $l\in \mathbb{N} $
\begin{equation}\label{projection_approximation}
\left\lVert u_N-u\right\rVert _{L^\infty(0,T;H^k)}\leq Ch^{l-k}\left\lVert u\right\rVert _{L^\infty(0,T;H^l)},\quad \forall 0\leq k\leq l.
\end{equation}

In fact, the Fourier projection estimates do not automatically preserve the positivity of $1+u_N$ and $1-u_N$; however, we could enforce the phase separation property that 
\begin{equation}
\left\lVert u_N\right\rVert _\infty \leq 1-\delta,
\end{equation}
if $h$ is taken sufficiently small. 
It is clear that $\int_{\Omega} u_N(\bm{x},t_n)\d \bm{x}= \int_{\Omega} u(\bm{x},t_n)\d \bm{x}$, for any $n\in \mathbb{N}$, due to the fact that $u_N$ is the Fourier projection of $u$, and thus,
\begin{equation*}
\int_{\Omega}u_N(\bm{x},t_n) \d\bm{x}=\int_{\Omega}u(\bm{x},t_n)\d\bm{x}=\int_{\Omega}u(\bm{x},t_{n-1})\d\bm{x}=\int_{\Omega}u_N(\bm{x},t_{n-1})\d\bm{x},\quad \forall n\in\mathbb{N}, 
\end{equation*}
in which the second step is based on the fact that the exact solution $u$ is mass conservative at the continuous level. On the other hand, Lemma \ref{first_order_mass_conservation}  and  Lemma \ref{second_order_mass_conservation}  both indicate that  the numerical solution is mass conservative at the discrete level:
\begin{equation}\label{mass_conservative}
\overline{u^n}=\overline{u^{n-1}}\quad \forall n\in \mathbb{N}.
\end{equation}
Meanwhile, since $u_N\in \mathcal{B}^N$, it always holds that 
\begin{equation*}
\int_{\Omega}u_N(\bm{x},t_n)d\bm{x}=h^2 \sum_{i,j=1}^{M}u_N(x_i,y_j,t_n)=h^2\sum_{i,j=1}^{M}(I^h u_N^n)_{i,j}, 
\end{equation*}
where
$\phi_N^n$ = $\phi_N(\bm{x},t_n)$.
The mass conservative property is available at the discrete level:
\begin{equation*}
\overline{I^h u_N^n}=\overline{I^h u_N^{n-1}}.
\end{equation*}
As indicated before, we use the mass conservative projection for the initial data: $u^0=I^hu_N(\bm{x},0)$, that is,
\begin{equation}\label{initial_value}
(u^0)_{i,j}:=u_N(x_i,y_j,0).
\end{equation}
The error grid function is defined as
\begin{equation*}
e^n:=u^n-I^hu_N^n ,\quad \widetilde{e}^n=\widetilde{u}^n-I^h u_N^n.
\end{equation*}
Therefore, it follows that 
\begin{equation}\label{3_2}
\overline{e^n}=0,  
\end{equation}
due to the fact that 
\begin{equation}
\overline{I^h u_N^n}=\overline{I^h u_N^0}=\overline{u^0}=\overline{u^n} \quad \forall n\geq0.
\end{equation}

Note that the Fourier projection of the exact solution has to be taken at the initial time step as in \eqref{initial_value}, instead of a pointwise interpolation of the exact initial value, to ensure the zero-property of the error grid function $e^n$ at a discrete level.

\begin{lemma}\label{lagrange_error}
For the error functions $e^n$ and $\widetilde{e}^{n}$, it holds that 
\begin{equation}
\left\lVert e^{n}\right\rVert ^2+\left\lVert e^{n}-\widetilde{e}^{n}\right\rVert ^2\leq \left\lVert \widetilde{e}^n\right\rVert ^2.
\end{equation}
\end{lemma}

\begin{proof}
$n=0$ is trivial. For $n\geq 1$, we derive from \eqref{corrector} that 
\begin{equation}\label{3_3}
e^{n}-\widetilde{e}^{n}=\lambda_h^{n}g'(u^{n})+\xi_h^n.
\end{equation}
Taking the discrete $L^2$ inner product with \eqref{3_3}  by $2e^n$ yields 
\begin{equation}\label{1_5}
\left\lVert e^n\right\rVert ^2 -\left\lVert \widetilde{e}^{n}\right\rVert ^2 +\left\lVert e^{n}-\widetilde{e}^{n}\right\rVert ^2=2\langle \lambda_h^ng'(u^n),e^{n}\rangle +2\xi_h^n \langle e^n,1\rangle .
\end{equation}
The last term on the right side of \eqref{1_5} can be eliminated due to the fact that $\overline{e^n}=0$. 
On the other hand, the KKT condition \eqref{1_4_b} indicates that 
\begin{equation}\label{leq}
\begin{aligned}
2\langle \lambda_h^ng'(u^n),e^n\rangle &=2 \langle \lambda_h^ng'(u^n),u^n-u_N(t_n)\rangle-2\langle \lambda_h ^n ,g(u^n)\rangle \\
&=-2\langle \lambda_h^n,g'(u^n)(u_N(t_n)-u^n)+g(u^n)\rangle.
\end{aligned}
\end{equation}
We omit the operator $I^h$ which should appear in front of $u_N(t_n)$, for simplicity.
By Taylor expansion, we have 
\begin{equation}\label{taylor}
g(u_N(t_n))=g(u^n)+g'(u^n)(u_N(t_n)-u^n)+\frac{1}{2}g''(\zeta _h)(u_N(t_n)-u^n)^2, 
\end{equation}
with $g''(\zeta_h)=-2$.
Substituting \eqref{taylor} into \eqref{leq}, we obtain 
\begin{equation}
\begin{aligned}
&2\langle \lambda_h^ng'(u^n),e^n\rangle\\
&=-2\langle \lambda_h^n,g(u_N(t_n))-\frac{1}{2} g''(\zeta_h)(u_N(t_n)-u^n)^2\rangle\\
&=-2\langle \lambda_h^n, g(u_N(t_n))+(u_N(t_n)-u^n)^2  \rangle \\
&\leq 0,
\end{aligned}
\end{equation}
where we have used the fact that $\lambda_h^n\geq0$ at a point-wise level and $\left\lVert u_N(t_n) \right\rVert _\infty \leq 1-\delta $ as long as the spatial step size is taken sufficiently small. Combining the above inequality, we complete the proof.
\end{proof}

\subsection{Error estimate for the P-ETD1 method}

\begin{theorem}\label{convergence_theorem_ETD1}
Given a fixed time $T>0$, and suppose that the exact solution for NCH equation \eqref{nonlocal_eq} is smooth enough. If $\tau \leq \frac{1}{2}  $ and $h$ is small sufficiently,  the numerical solution generated by the P-ETD1 scheme has the following error estimate:
\begin{equation}
\left\lVert u^n-I^h u(t_n)\right\rVert \leq C(\tau+h^2),
\end{equation}
where the constant $C>0$ depends on $T$, $\epsilon$, $\kappa$, $\theta_c$ and $\theta$, but is independent of $\tau$ and $h$. 
\end{theorem}

\begin{proof}
It suffices to estimate the term $\left\lVert u^n -I^h u_N(t_n) \right\rVert$ under the projection approximation \eqref{projection_approximation}.
For the NCH equation, we can give a similar representation as follows: for a given $u_N(\bm{x},t_n)$, the solution $u_N(\bm{x},t_{n+1})=w(\bm{x},\tau)$ is given by the following system:
\begin{equation}\label{ODE_exact_solution}
\left\{
\begin{aligned}
&\frac{\partial w(s)}{\partial s}-\kappa \Delta w+\epsilon^2 \Delta ^2w+\sigma w=\mathcal{P}_N\widetilde{F}(u(\bm{x},t_n+s)), \quad s\in(0,\tau],\\
&w(\bm{x},0)=u_N(\bm{x},t_n),
\end{aligned}
\right.
\end{equation}
where 
\begin{equation}
\begin{aligned}
\widetilde{F}(u(\bm{x},t_n+s))&=\frac{\theta}{2} \Delta[\ln(1+u(\bm{x},t_n+s))-\ln(1-u(\bm{x},t_n+s))]-\theta_c \Delta u(\bm{x},t_n+s)\\
&\quad -k\Delta u(\bm{x},t_n+s)+\sigma \overline{u(\bm{x},t_n+s)}.
\end{aligned}
\end{equation} 
Denote by  $e_1(s)=W_1(s)-I^hw(s)$ and the difference between \eqref{ODE_ETD1} and  \eqref{ODE_exact_solution} leads to 
\begin{equation}\label{ODE_error_ETD1}
\left\{
\begin{aligned}
&\frac{\mathrm{d} e_1(s)}{\mathrm{d}s}+L_h e_1(s)=F(u^n)-F(u_N(t_n))+R_1(s), \quad s\in(0,\tau),\\
&e_1(0)=u^n-u_N(t_n)=e^n,
\end{aligned}
\right.
\end{equation}
where $R_1(s)$ is the truncation error satisfying
\begin{equation}\label{truncation1}
\left\lVert R_1(s)\right\rVert \leq C_e (\tau+h^2), \quad s\in(0,\tau) ,
\end{equation}
where $C_e$ depends on $u$, $\epsilon$, $\kappa$, $\theta_c$ and $\theta$. Note that we drop the operator $I^h$, which should appear in front of $u_N(t_n)$, for simplicity.
By using the variation-of-constants formula to  \eqref{ODE_error_ETD1} and setting $t=\tau$ leads to 
\begin{equation}\label{error_interger}
\widetilde{e}^{n+1}=\varphi_0(\tau L_h)e^n+\tau \varphi_1(\tau L_h) (F(u^n)-F(u_N(t_n)))+\int_{0}^{\tau}\mathrm{e}^{-(\tau-s)L_h}R_1(s)  \,\d s.
\end{equation}
Acting $I+\tau L_h$ on both sides of \eqref{error_interger} and taking the discrete $L^2$ inner product of the resulted equation with $\widetilde{e}^{n+1}$ yield 
\begin{equation}\label{3_5}
\begin{aligned}
&\left\lVert \widetilde{e}^{n+1}\right\rVert ^2+\epsilon^2 \tau \left\lVert \Delta_h \widetilde{e}^{n+1}\right\rVert ^2+\sigma \left\lVert \widetilde{e}^{n+1}\right\rVert ^2-\kappa \tau \langle \Delta_h \widetilde{e}^{n+1},\widetilde{e}^{n+1}\rangle  \\
=&\langle (I+\tau L_h)\varphi_0(\tau L_h),\widetilde{e}^{n+1}\rangle +\tau \langle (I+\tau L_h)\varphi_1(\tau L_h)(F(u^n)-F(u_N(t_n))), \widetilde{e}^{n+1}\rangle \\
&+\int_{0}^{\tau}\langle (I+\tau L_h)\mathrm{e}^{-(\tau-s)L_h}R_1(s),\widetilde{e}^{n+1} \rangle \d s.
\end{aligned}
\end{equation}
For the first term in the right hand \eqref{3_5}, by Lemma \ref{property2} and Schwartz's inequality, we have 
\begin{equation}\label{3_8}
\langle (I+\tau L_h)\varphi_0(\tau L_h)e^n,\widetilde{e}^{n+1}\rangle \leq \left\lVert (I+\tau L_h)\varphi_0(\tau L_h) \right\rVert \left\lVert e^n\right\rVert \left\lVert \widetilde{e}^{n+1}\right\rVert \leq \frac{1}{2} \left\lVert e^n\right\rVert ^2+\frac{1}{2} \left\lVert \widetilde{e}^{n+1}\right\rVert ^2 .
\end{equation} 
Note that $\overline{e^n}=0$, and thus, we have
\begin{equation}
F(u^n)-F(u_N(t_n))=\Delta_h \left(f(u^n)-f(u_N(t_n))\right).
\end{equation}
Under the assumption that $\left\lVert u_N(t_n)\right\rVert _\infty \leq 1-\delta$ and $\left\lVert u^n\right\rVert _\infty \leq 1-\delta$, and by Taylor expansion, we obtain 
\begin{equation}\label{nonlinear_estimate}
\left\lVert f(u^n)-f(u_N(t_n))\right\rVert =\left\lVert f'(\eta_n)e^n\right\rVert \leq \left\lvert \frac{\theta}{1-\eta_n^2} -\theta_c-\kappa \right\rvert \left\lVert e^n\right\rVert := C_1\left\lVert e^n\right\rVert ,
\end{equation}  
where the constant $C_1$ depends on $\delta$, $\theta$, $\theta_c$ and $\kappa$.
Combining Lemma \ref{property2}, Eq. \eqref{nonlinear_estimate} and Schwartz's inequality, the second term in the right hand of \eqref{3_5} can be bounded as
\begin{equation}\label{4_0}
\begin{aligned}
&\tau \langle (I+\tau L_h)\varphi_1(\tau L_h)(F(u^n)-F(u_N(t_n))), \widetilde{e}^{n+1}\rangle\\ 
&=\tau \langle (I+\tau L_h)\varphi_1(\tau L_h)(f(u^n)-f(u_N(t_n))),\Delta_h \widetilde{e}^{n+1}\rangle\\
&\leq \tau C_1 \left\lVert (I+\tau L_h)\varphi_1(\tau L_h) \right\rVert  \left\lVert e^n\right\rVert \left\lVert \Delta_h \widetilde{e}^{n+1}\right\rVert\\
& \leq \frac{C_1^2}{\epsilon^2} \tau \left\lVert e^n\right\rVert ^2+\epsilon^2 \tau \left\lVert \Delta_h \widetilde{e}^{n+1}\right\rVert ^2  .
\end{aligned}
\end{equation}
The last term of the right hand side of \eqref{3_5} can be estimated as 
\begin{equation}\label{4_1}
\begin{aligned}
\int_{0}^{\tau}\langle (I+\tau L_h)\mathrm{e}^{-(\tau-s)L_h}R_1(s),\widetilde{e}^{n+1} \rangle \d s&\leq \tau\left\lVert (I+\tau L_h)\mathrm{e}^{-(\tau-s)L_h}\right\rVert  \left\lVert R_1(s)\right\rVert \left\lVert \widetilde{e}^{n+1}\right\rVert \\
& \leq \frac{1}{2} \tau  \sup_{s\in (0,\tau)}\left\lVert R_1(s)\right\rVert ^2+\frac{1}{2} \tau \left\lVert \widetilde{e}^{n+1}\right\rVert ^2 ,
\end{aligned}
\end{equation}
where we have used Lemma \ref{property2} in the last inequality.
Substituting \eqref{3_8}-\eqref{4_1} into \eqref{3_5}, we obtain 
\begin{equation}
\frac{1}{2}(1-\tau ) \left\lVert \widetilde{e}^{n+1}\right\rVert ^2\leq \frac{1}{2}\left\lVert {e}^{n}\right\rVert ^2+ \frac{C_1^2}{\epsilon^2}\tau \left\lVert e^n\right\rVert ^2+\frac{1}{2} \tau  \sup_{s\in(0,\tau)}\left\lVert R_1(s)\right\rVert ^2.
\end{equation}
Using Lemma \ref{lagrange_error}, we have 
\begin{equation}
\frac{1}{2}(1-\tau ) \left\lVert {e}^{n+1}\right\rVert ^2\leq \frac{1}{2}\left\lVert {e}^{n}\right\rVert ^2+ \frac{C_1^2}{\epsilon^2}\tau \left\lVert e^n\right\rVert ^2+\frac{1}{2} \tau  \sup_{s\in(0,\tau)}\left\lVert R_1(s)\right\rVert ^2,
\end{equation}
that is 
\begin{equation}
(1-\tau ) \left\lVert {e}^{n+1}\right\rVert ^2\leq (1+\frac{2C_1^2}{\epsilon^2}\tau ) \left\lVert e^n\right\rVert ^2+\tau  \sup_{s\in(0,\tau)}\left\lVert R_1(s)\right\rVert ^2.
\end{equation}
When $\tau \leq \frac{1}{2}$, we obtain
\begin{equation}
\frac{1+ \frac{2C_1^2}{\epsilon^2} \tau}{1- \tau} \leq  1+\frac{4C_1^2+2\epsilon^2}{\epsilon^2}  \tau=1+C_2 \tau,
\end{equation}
hence
\begin{equation}
\left\lVert e^{n+1}\right\rVert ^2\leq (1+4C_2\tau) \left\lVert e^n\right\rVert ^2+2C_e^2\tau (\tau+h^2)^2. 
\end{equation}
Using the discrete Gronwall inequality, we obtain
\begin{equation}
\left\lVert e^{n+1}\right\rVert ^2\leq 2C_e^2\mathrm{e}^{4C_2 T}(\tau+h^2)^2,
\end{equation}
which completes the proof.
\end{proof}

\subsection{Error estimate for the P-ETDRK2 scheme }

\begin{theorem}
Given a fixed time $T>0$, and suppose that the exact solution for NCH equation \eqref{nonlocal_eq} is smooth enough. If $\tau \leq \frac{1}{2}  $ and $h$ is small sufficiently,  the numerical solution generated by the P-ETDRK2 scheme has the following error estimate:
\begin{equation}
\left\lVert u^n-I^hu(t_n)\right\rVert \leq C(\tau ^2+h^2),
\end{equation}
where the constant $C>0$ depends on $T$, $u$, $\epsilon$, $\kappa$, $\theta_c$ and $\theta$, but is independent of $\tau$  and $h$.
\end{theorem}

\begin{proof}
Similar to the error analysis of the P-ETD1 scheme, we only need to estimate $\left\lVert u^n-I^h u_N(t_n)\right\rVert$.
Let $e_2(s)=W_2(s)-I^hw(s)$. Taking the difference between \eqref{ODE_ETDRK2} and \eqref{ODE_exact_solution} yields the error equation for $s\in (0,\tau]$:
\begin{equation}\label{ODE_error_ETDRK2}
\left\{
\begin{aligned}
&\frac{\mathrm{d} e_2(s)}{\mathrm{d}s}+L_h e_2(s)=(1-\frac{s}{\tau} )(F(u^n)-F(u_N(t_n)))+\frac{s}{\tau}  (F(u^{n+1}_{mid})-F(u_N(t_{n+1}))) +R_2(s),\\
&e_2(0)=u^n-u_N(t_n)=e^n,
\end{aligned}
\right.
\end{equation}
where the truncation error $ R_2(s)$ satisfying  
\begin{equation}\label{truncation2}
\left\lVert R_2(s)\right\rVert \leq C_e(\tau^2+h^2), \quad s \in (0,\tau),
\end{equation}
under the assumption that $\left\lVert u_N\right\rVert _\infty \leq 1-\delta$.
Using variation-of-constants formula  to \eqref{ODE_error_ETDRK2} and setting $t=\tau$ leads to 
\begin{equation}\label{5_4}
\begin{aligned}
\widetilde{e}^{n+1}=&\varphi_0(\tau L_h)e^n +\tau (\varphi_1-\varphi_2)(\tau L_h) (F(u^n)-F(u_N(t_n)))+\tau \varphi_2(\tau L_h) (F(u^{n+1}_{\text{mid}})-F(u_N(t_{n+1})))\\
&+\int_{0}^{\tau }\mathrm{e}^{-(\tau-s)L_h}R_2(s) \,\d s .
\end{aligned}
\end{equation}
Acting $I+\tau L_h$ on both sides of \eqref{5_4} and taking the discrete $L^2$ inner product of the resulted equation with $\widetilde{e}^{n+1}$ give us 
\begin{equation}\label{5_5}
\begin{aligned}
&\left\lVert \widetilde{e}^{n+1}\right\rVert ^2+\epsilon^2 \tau  \left\lVert \Delta_h \widetilde{e}^{n+1}\right\rVert ^2 -\kappa \tau \langle \Delta_h \widetilde{e}^{n+1},\widetilde{e}^{n+1}\rangle +\sigma \left\lVert \widetilde{e}^{n+1}\right\rVert ^2\\
&=\langle (I+\tau L_h)\varphi_0(\tau L_h)e^n, \widetilde{e}^{n+1}\rangle +\tau \langle (I+\tau L_h)(\varphi_1-\varphi_2)(\tau L_h)(F(u^n)-F(u_N(t_n))),\widetilde{e}^{n+1}\rangle\\
&\quad +\tau \langle (I+\tau L_h)\varphi_2(\tau L_h)(F(u^{n+1}_{\text{mid}})-F(u_N(t_{n+1}))), \widetilde{e}^{n+1}\rangle \\
&\quad +\int_{0}^{\tau} \langle (I+\tau L_h)\mathrm{e}^{-(\tau-s)L_h}R_2(s),\widetilde{e}^{n+1}\rangle  \,\d s .
\end{aligned}
\end{equation}
Similar to Eq.\eqref{3_8}, the first term on the right hand of \eqref{5_5} indicates  that 
\begin{equation}\label{5_6}
\langle (I+\tau L_h)\varphi_0(\tau L_h)e^n,\widetilde{e}^{n+1}\rangle \leq \frac{1}{2} \left\lVert e^n\right\rVert ^2+\frac{1}{2}  \left\lVert \widetilde{e}^{n+1}\right\rVert^2.
\end{equation} 
By \eqref{nonlinear_estimate}, Lemma \ref{property2} and Schwartz's inequality we have 
\begin{equation}\label{5_7}
\begin{aligned}
&\tau \langle (I+\tau L_h)(\varphi_1(\tau L_h)-\varphi_2(\tau L_h))(F(u^n)-F(u_N(t_n))), \widetilde{e}^{n+1}\rangle\\
&=\tau \langle (I+\tau L_h)(\varphi_1(\tau L_h)-\varphi_2(\tau L_h))(f(u^n)-f(u_N(t_n))), \Delta_h \widetilde{e}^{n+1}\rangle\\
&\leq \tau C_1 \left\lVert e^n\right\rVert \left\lVert \Delta_h \widetilde{e}^{n+1}
\right\rVert \\
& \leq \frac{C_1^2}{2\epsilon^2}  \tau \left\lVert e^n\right\rVert ^2+\frac{\epsilon^2}{2}\tau \left\lVert \Delta_h\widetilde{e}^{n+1}\right\rVert ^2.
\end{aligned}
\end{equation}
For the third term on the right side of \eqref{5_5}, $u^{n+1}_{\text{mid}}$ is the  numerical solution for P-ETD1 scheme with $\left\lVert u^{n+1}_{\text{mid}} \right\rVert _\infty \leq 1-\delta $. In \eqref{4_1}, if we adopt the estimate 
\begin{equation}
\left\lVert R_1(s)\right\rVert \left\lVert \widetilde{e}^{n+1}\right\rVert  \leq \tau^2 \sup_{s\in(0,\tau)}\left\lVert R_1(s)\right\rVert ^2+\frac{1}{4}  \left\lVert \widetilde{e}^{n+1}\right\rVert ^2, 
\end{equation}
then we have the following error estimate 
\begin{equation}\label{5_9}
\begin{aligned}
\left\lVert u_{\text{mid}}^{n+1}-u_N(t_{n+1})\right\rVert ^2&\leq (2+\frac{2C_1^2}{\epsilon^2}\tau )\left\lVert e^n\right\rVert ^2+2\tau^2 \sup_{s\in(0,\tau)}\left\lVert R_1(s)\right\rVert ^2 \\
& \leq(2+\frac{4C_1^2}{\epsilon^2}\tau )\left\lVert e^n\right\rVert ^2+2C_e^2\tau^2 (\tau+h^2)^2.
\end{aligned} 
\end{equation}
Thus, Schwartz's inequality and above estimates \eqref{5_9}, we have 
\begin{equation}\label{6_0}
\begin{aligned}
&\tau \langle (I+\tau L_h)\varphi_2(\tau L_h)(F(u^{n+1}_{\text{mid}})-F(u_N(t_{n+1}))), \widetilde{e}^{n+1}\rangle \\
&=\tau \langle (I+\tau L_h)\varphi_2(\tau L_h)(f(u^{n+1}_{\text{mid}})-f(u_N(t_{n+1}))), \Delta _h\widetilde{e}^{n+1}\rangle \\
&\leq C_1\tau \left\lVert u^{n+1}_{\text{mid}}-u_N(t_{n+1})\right\rVert \left\lVert \Delta_h \widetilde{e}^{n+1}\right\rVert  \\
&\leq\frac{C_1^2}{2\epsilon^2} \tau \left\lVert u^{n+1}_{\text{mid}}-u_N(t_{n+1})\right\rVert ^2+\frac{\epsilon^2 \tau}{2}\left\lVert \Delta_h \widetilde{e}^{n+1}\right\rVert ^2 \\
&\leq \frac{C_1^2}{2\epsilon^2}(2+\frac{2C_1^2}{\epsilon^2}\tau )\tau \left\lVert e^n\right\rVert ^2 +\frac{2C_1^2C_e}{\epsilon^2}\tau^3 (\tau+h^2) ^2+\frac{\epsilon^2 \tau}{2}\left\lVert \Delta_h \widetilde{e}^{n+1}\right\rVert ^2 \\
&\leq \frac{C_1^2}{\epsilon^2}(2+\frac{2C_1^2}{\epsilon^2}\tau )\tau \left\lVert e^n\right\rVert ^2 +\frac{2C_1^2C_e}{\epsilon^2}\tau (\tau^2+h^2) ^2+\frac{\epsilon^2 \tau}{2}\left\lVert \Delta_h \widetilde{e}^{n+1}\right\rVert ^2.
\end{aligned}
\end{equation}
For the last term of the right hand of \eqref{5_5}, using the estimate for the truncation error \eqref{truncation2} we conclude that 
\begin{equation}\label{6_1}
\begin{aligned}
\int_{0}^{\tau} \langle (I+\tau L_h)\mathrm{e}^{-(\tau-s)L_h} R_2(s),\widetilde{e}^{n+1}\rangle  \,\d s &\leq \tau \left\lVert (I+\tau L_h)\mathrm{e}^{-(\tau-s)L_h}\right\rVert \left\lVert R_2(s)\right\rVert \left\lVert \widetilde{e}^{n+1}
\right\rVert   \\
&\leq \frac{1}{2} \tau \sup_{s\in(0,\tau)}\left\lVert R_2(s)\right\rVert ^2+\frac{1}{2} \tau \left\lVert \widetilde{e}^{n+1}\right\rVert ^2\\
&\leq  C_e^2 \tau (\tau^2+h^2)^2+\frac{1}{2}\tau \left\lVert \widetilde{e}^{n+1}\right\rVert ^2.
\end{aligned}
\end{equation}
Substituting \eqref{5_6}-\eqref{6_1} to \eqref{5_5}, we obtain 
\begin{equation}
(1-\tau)\left\lVert \widetilde{e}^{n+1}\right\rVert ^2\leq(1+\frac{C_1^2}{\epsilon^2}\tau+\frac{C^2}{\epsilon^2}(2+\frac{2C_1^2}{\epsilon^2}\tau ) \tau  ) \left\lVert e^n\right\rVert ^2 +C_3\tau (\tau^2+h^2)^2,
\end{equation}
where $C_3=\max\{\frac{4C_1^2C_e^2}{\epsilon^2},2C_e^2\}$.
Using Lemma \ref{lagrange_error}, we have 
\begin{equation}
(1-\tau)\left\lVert {e}^{n+1}\right\rVert ^2\leq(1+\frac{C_1^2}{\epsilon^2}\tau+\frac{C_1^2}{\epsilon^2}(2+\frac{2C_1^2}{\epsilon^2}\tau ) \tau  ) \left\lVert e^n\right\rVert ^2 +C_3\tau (\tau^2+h^2)^2.
\end{equation}
When $\tau \leq \frac{1}{2} $, similar to the last paragraph in the proof of Theorem \ref{convergence_theorem_ETD1}, applying the discrete Gronwall's inequality yields
\begin{equation}
    \left\lVert e^{n+1}\right\rVert ^2 \leq C(\tau^2+h^2)^2,
\end{equation}
which completes the proof.
\end{proof}

\section{Numerical experiments}\label{experiments}

In this section, we perform several numerical examples to demonstrate the performance of our proposed schemes. 

\subsection{Convergence in time}

We first verify the convergence order for the proposed P-ETD schemes. Let us set $\epsilon=0.02$, $\theta=0.8$, $\theta_c=1.6$, $\delta=0.05$, $\kappa=1$ and take a smooth initial value
\begin{equation}
u^0(x,y)=0.1\sin{2\pi x}\sin{2\pi y}.
\end{equation}
The temporal convergence tests are conducted by fixing the spatial mesh size $h=1/512$. We compute the numerical solution at $T=0.02$ using the P-ETD1 and P-ETDRK2 scheme with various time step size $\tau=2^{-k}\times  10^{-4}$, $k=0,1,\cdots ,4$. To compute the numerical errors, we treat P-ETDRK2 solution obtained by $\tau=10^{-6}$ as the benchmark solution. The discrete $L^2$-norm errors and corresponding convergence rates computed by P-ETD1 scheme  and P-ETDRK2 scheme are presented in 
Table \ref{convergencePETD1} and Table \ref{convergencePETDRK2}, respectively. The first-order temporal accuracy for P-ETD1 scheme and second-order for P-ETDRK2 scheme are observed as expected.

\begin{table}[h]
\centering
\caption{The discrete $L^2$-norm errors and convergence rates produced by the proposed P-ETD1 scheme for NCH equation }\label{convergencePETD1}
\begin{tabular}{l*{6}{c}}
\bottomrule
$\tau$ & $10^{-4}$ & $10^{-4}/2$ & $10^{-4}/4$ & $10^{-4}/8$ & $10^{-4}/16$ \\
\midrule
$L^2$ Error  & 1.6724E-01
&	9.5485E-02
&5.0341E-02
&	2.5032E-02
&1.1642E-02
\\
Rate &- &0.8086  &0.9236 &1.0079 &1.1045 \\
\bottomrule
\end{tabular}

\end{table}

\begin{table}[h]
\centering
\caption{The discrete $L^2$-norm errors and convergence rates produced by the proposed P-ETDRK2 scheme for NCH equation }\label{convergencePETDRK2}
\begin{tabular}{l*{6}{c}}
\bottomrule
$\tau$ & $10^{-4}$ & $10^{-4}/2$ & $10^{-4}/4$ & $10^{-4}/8$ & $10^{-4}/16$ \\
\midrule
$L^2$ Error  & 1.5449E-02
&	4.1488E-03
&  1.0713E-03
& 2.7083E-04
& 6.6914E-05
&	\\
Rate  &-& 1.8968 & 1.9534 & 1.9839 & 2.0170\\
\bottomrule
\end{tabular}

\end{table}

\subsection{Comparisons with the P-ETD schemes and the classical ETD schemes }

In the following numerical experiments, we compare the proposed P-ETD schemes with the classical ETD schemes by focusing on the evolutions of supremum norms for the numerical solution. The dynamic process considered usually needs a long-time evolution to reach the steady state; here we conduct simulation in a short time interval for the comparison among these schemes.

Let us consider the problem with $\epsilon=0.02$. We adopt the uniform spatial mesh size $h=1/512$, time step size $\tau=0.1$ and give the initial value 
\begin{equation}
u^0(x,y)=0.2+0.05\text{rand}(x,y),
\end{equation}
where the function $\text{rand}(x,y)$ is a uniform distributed random function with values in $(-1,1)$.
We set $\theta_c=1.6$, $\theta=0.8$, $\delta =0.05$,  $\sigma=30$ and compute the numerical solution by using the P-ETD schemes and the classic ETD schemes (ETD1 and ETDRK2). Figure \ref{comparison_ETD2} shows the evolutions of the supremum norms and the mass increment computed by the classic ETD1 and P-ETD1 schemes. It is observed that P-ETD1 scheme preserves the numerical solution between $-1$ and $1$ peer and holds the mass conservation property. However, the supremum norm of the classic ETD1 schemes evolves beyond 1 at some time and once the numerical solution exceeds $1$, complex numerical solutions will appear due to the existence of the logarithmic term and give the completely wrong dynamics. The number of iterations at the correction step obtained by P-ETD1 scheme and corresponding Lagrange multiplier $\lambda_h$ are presented in Figure \ref{lagrange_multiplier1}. The Lagrange multiplier $\lambda_h$ is not large, but it plays an important role in making the numerical solution located in $(-1,1)$.

Figure \ref{comparison_ETD3} and Figure \ref{lagrange_multiplier2} plot corresponding results computed by the second-order classic ETDRK2 scheme and P-ETDRK2 method. Again, only the P-ETDRK2 scheme preserves bound and mass conservation. The classic ETDRK2 scheme having the supremum norm beyond 1, leads to inaccurate dynamic process. 


\begin{figure}[h!]
\centering
{\includegraphics[width=8cm]{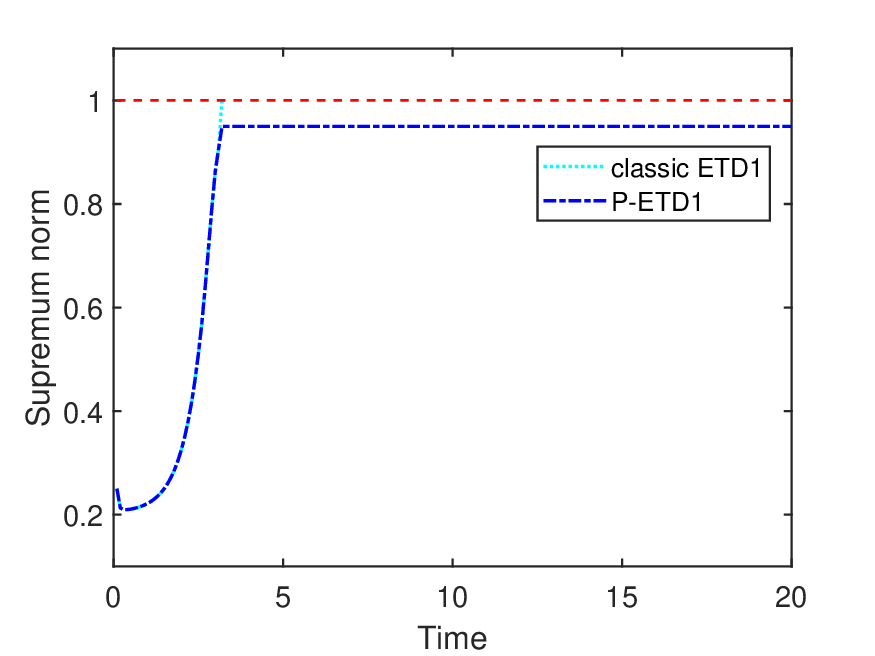}}
{\includegraphics[width=8cm]{
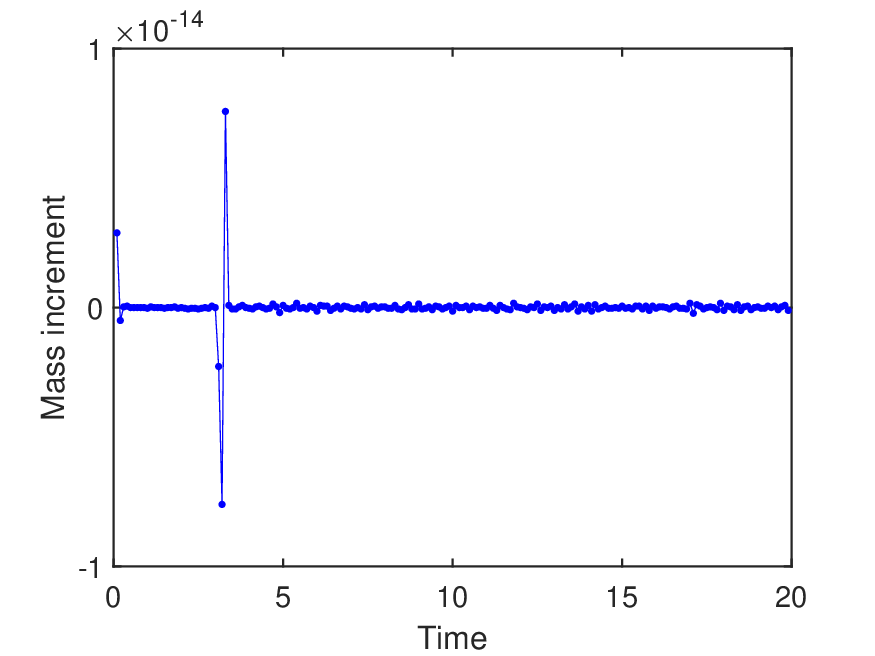}}
\caption{The evolutions of the supremum norms and mass increment of the numerical solution generated by the classic ETD1 scheme and P-ETD1 scheme with $\kappa=2$, $\tau=0.1$, $h=1/512$, $\epsilon=0.0 2$, $\theta_c=1.6$, $\theta=0.8$, $\sigma=30$ }
\label{comparison_ETD2}
\end{figure}

\begin{figure}[h!]
\centering
{\includegraphics[width=8.1cm]{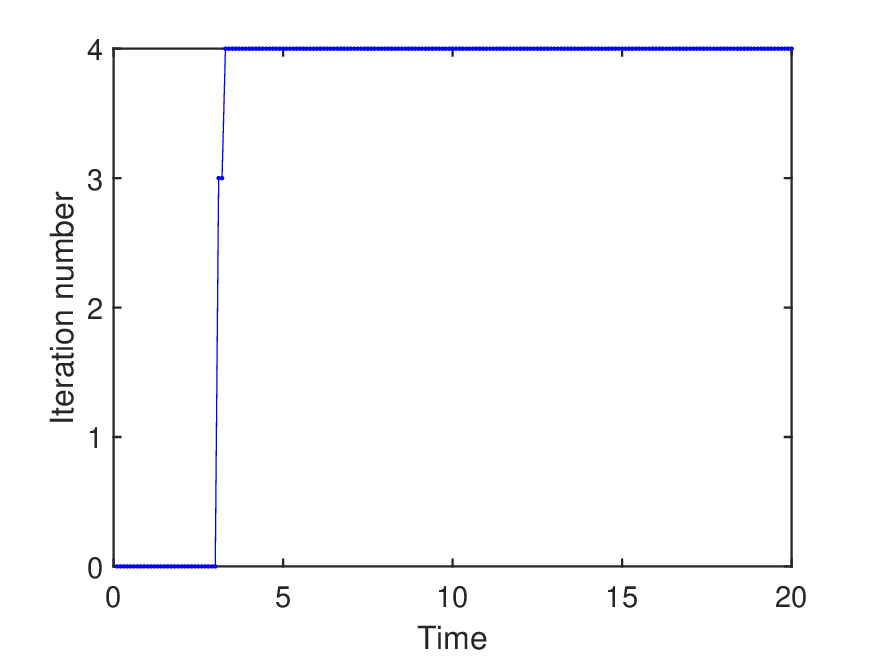 }}
{\includegraphics[width=8.1cm]{
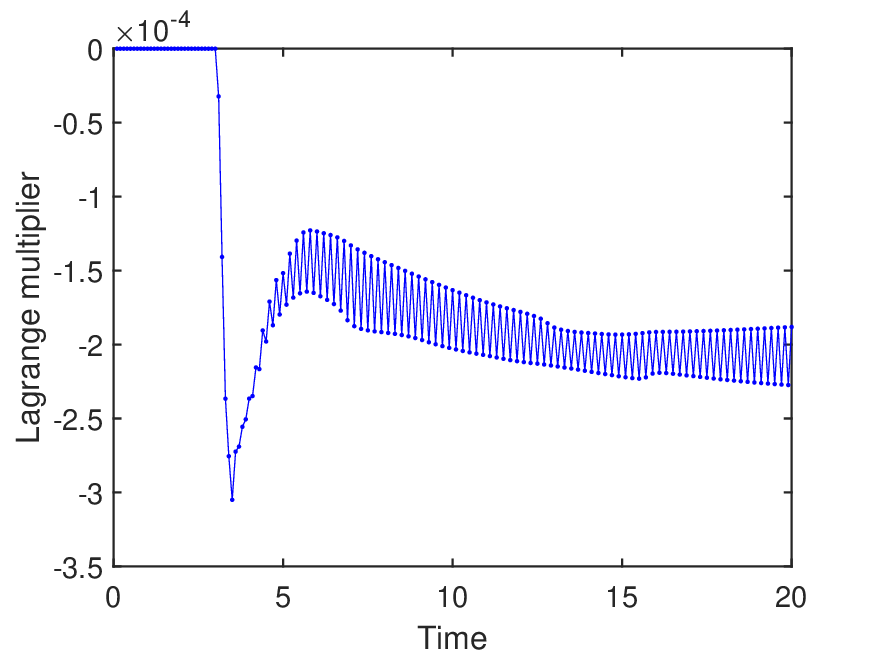
}}
\caption{The evolutions of the iteration number (left) and the Lagrange multiplier (right) generated by P-ETD1 scheme with $\kappa=2$, $\tau=0.1$, $h=1/512$, $\epsilon=0.0 2$, $\theta_c=1.6$, $\theta=0.8$, $\sigma=30$ }
\label{lagrange_multiplier1}
\end{figure}

\begin{figure}[h!]
\centering
{\includegraphics[width=8.1cm]{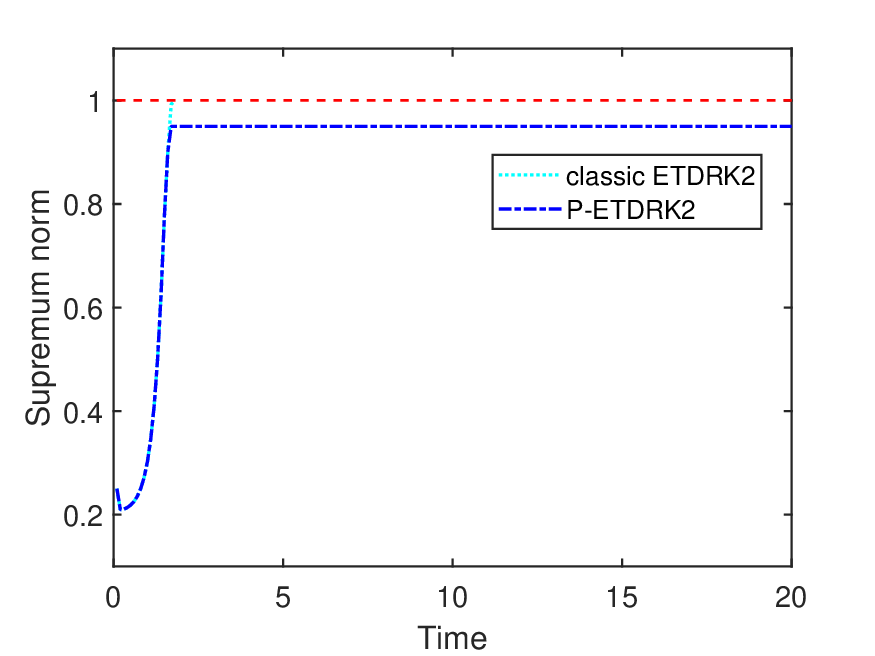}}
{\includegraphics[width=8.1cm]{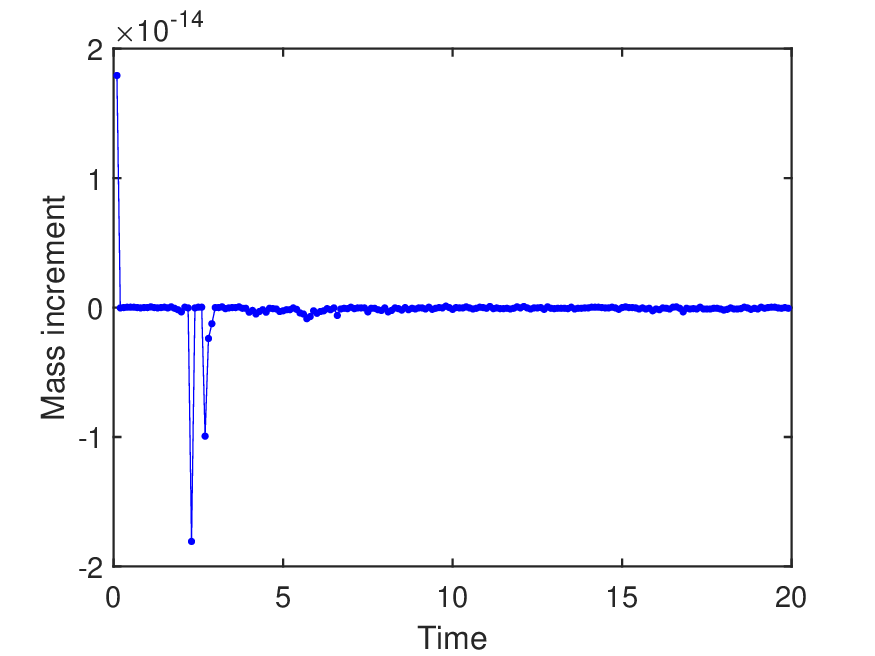}}
\caption{The evolutions of the supremum norms  and mass increment of the numerical solution generated by the classic ETDRK2 scheme and P-ETDRK2 scheme with $\kappa=2$, $\tau=0.1$, $h=1/512$, $\epsilon=0.02$, $\theta_c=1.6$, $\theta=0.8$, $\sigma=30$ }
\label{comparison_ETD3}
\end{figure}

\begin{figure}[h!]
\centering
{\includegraphics[width=8.1cm]{
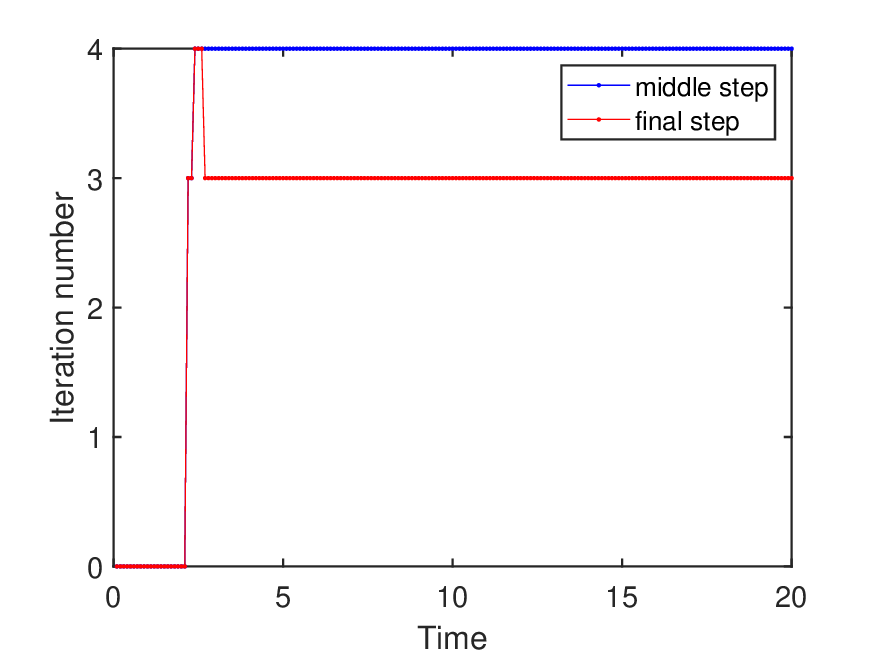}}
{\includegraphics[width=8.1cm]{
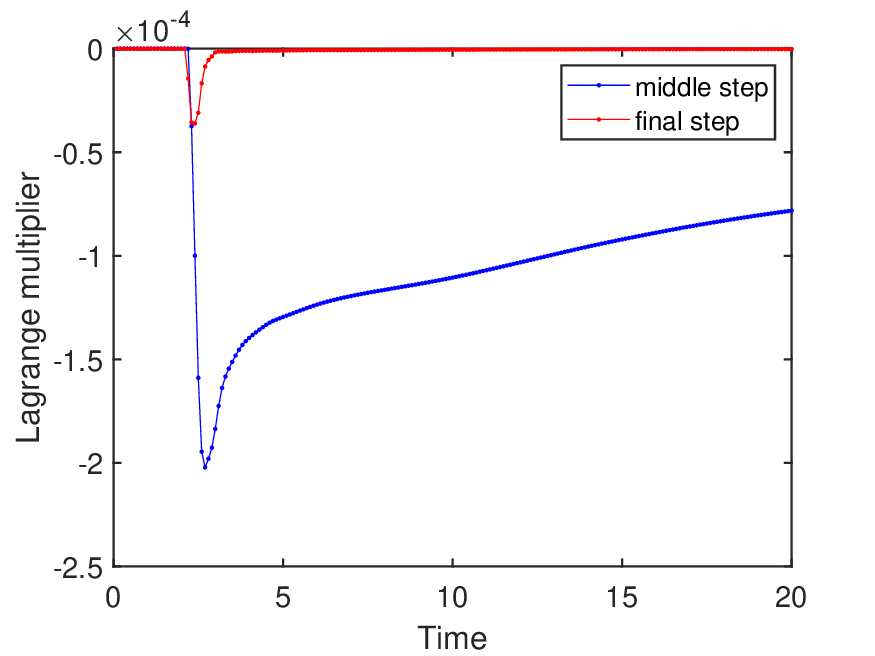}}
\caption{The evolutions of the iteration number (left) and the Lagrange multiplier (right) generated by P-ETDRK2 scheme with  $\kappa=2$, $\tau=0.1$, $h=1/512$, $\epsilon=0.0 2$, $\theta_c=1.6$, $\theta=0.8$, $\sigma=30$ }
\label{lagrange_multiplier2}
\end{figure}

\subsection{Long time Coarsening Dynamics Simulations} \label{coarsening_dynamics}

Now we study the coarsening dynamics driven by the NCH equation \eqref{nonlocal_eq} with $\epsilon=0.02$. The spatial mesh size is $h=1/512$ and the initial value is given by 
\begin{equation}\label{random_initial}
u^0(x,y)=0.3+0.05\text{rand}(x,y),
\end{equation}
where the function $\text{rand}(x,y)$ is a uniform distributed random function with values in $(-1,1)$. 
We adopt the P-ETDRK2 scheme with $\tau=0.1$ to simulate the long time coarsening process. The nonlocal constant is change to $\sigma=70$, and other parameters are the same as the convergence test.
The phase structures captured at $t=10$, $20$, $50$, $100$, $200$  and $2000$ are presented at Figure \ref{phase_structure}. 
It can be observed that, at the early stage of phase separation, the system is dominated by spherical structures with only a few stripe-like structures; as the evolution proceeds, the system eventually reaches equilibrium and exhibits a pure spherical phase.
The left picture given in Figure \ref{energy_mass} implies the energy dissipation during the whole phase transition process. The evolution of mass increment is plotted in the right graph of 
Figure \ref{energy_mass}, which states the mass conservation of the process. 

\begin{figure}[h!]
\centering
{\includegraphics[height=5cm,width=5.49cm]{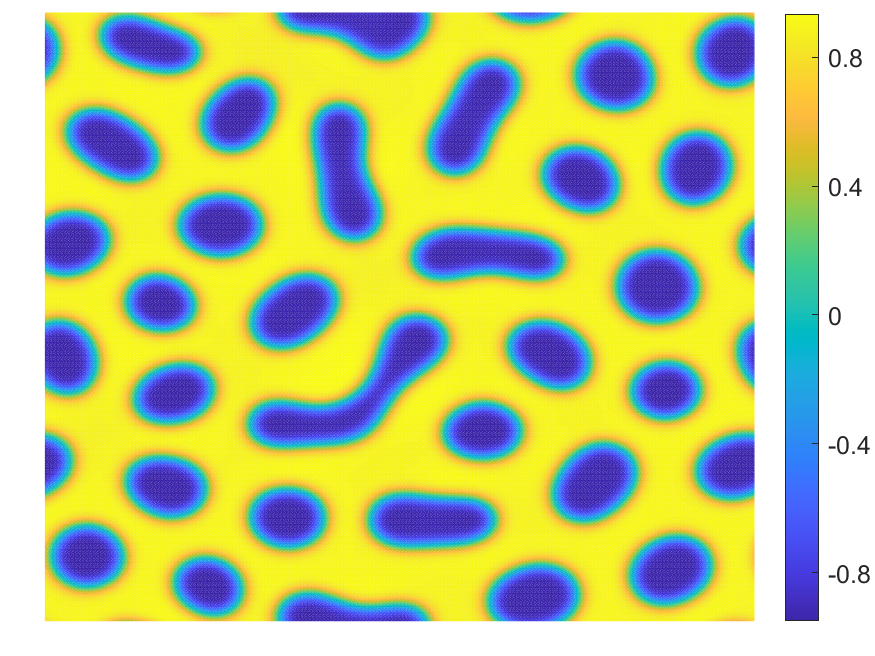}}
\hspace{-3mm}
{\includegraphics[height=5cm,width=5.49cm]{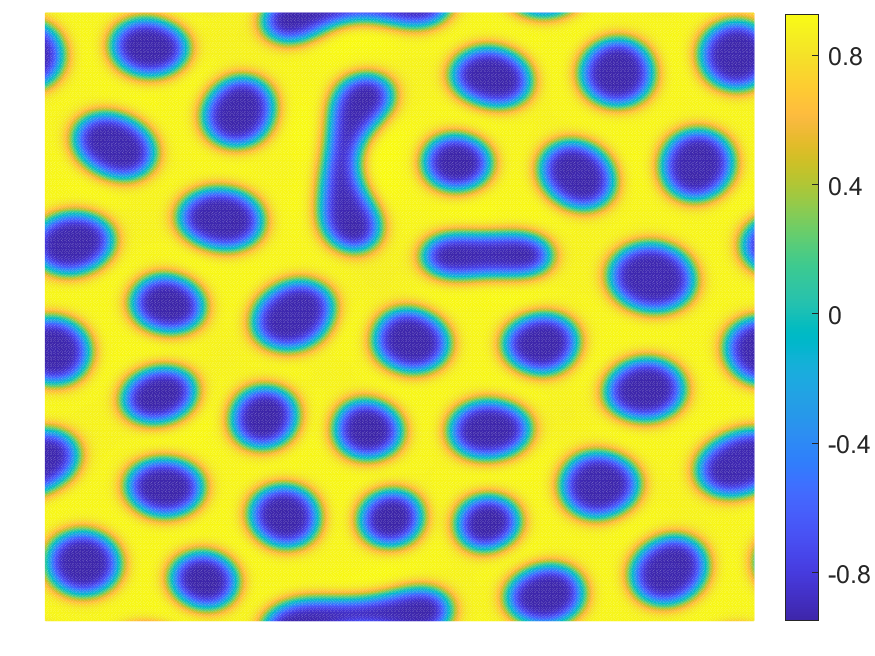}}
\hspace{-3mm}
{\includegraphics[height=5cm,width=5.49cm]{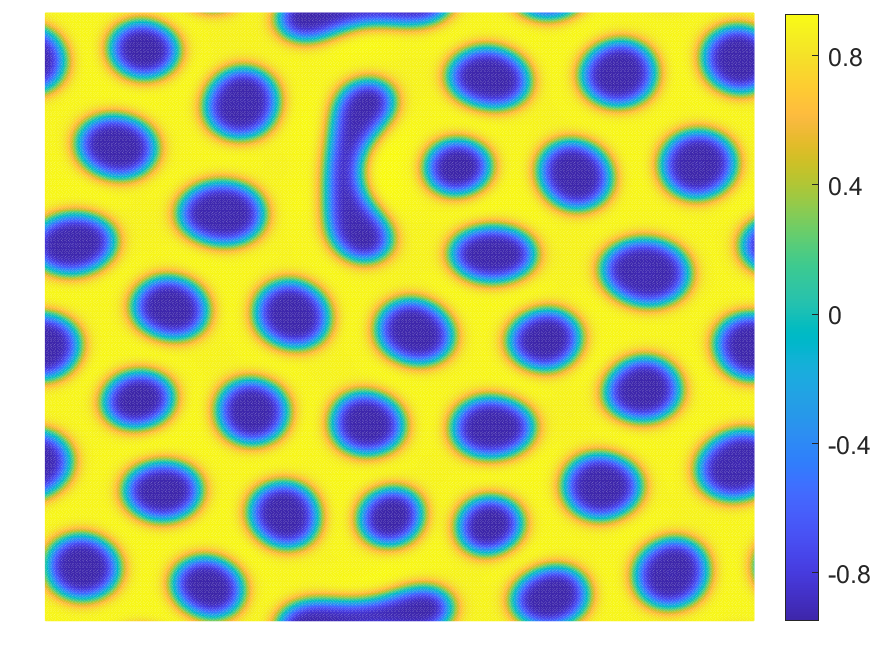}}
\hspace{-3mm}
{\includegraphics[height=5cm,width=5.49cm]{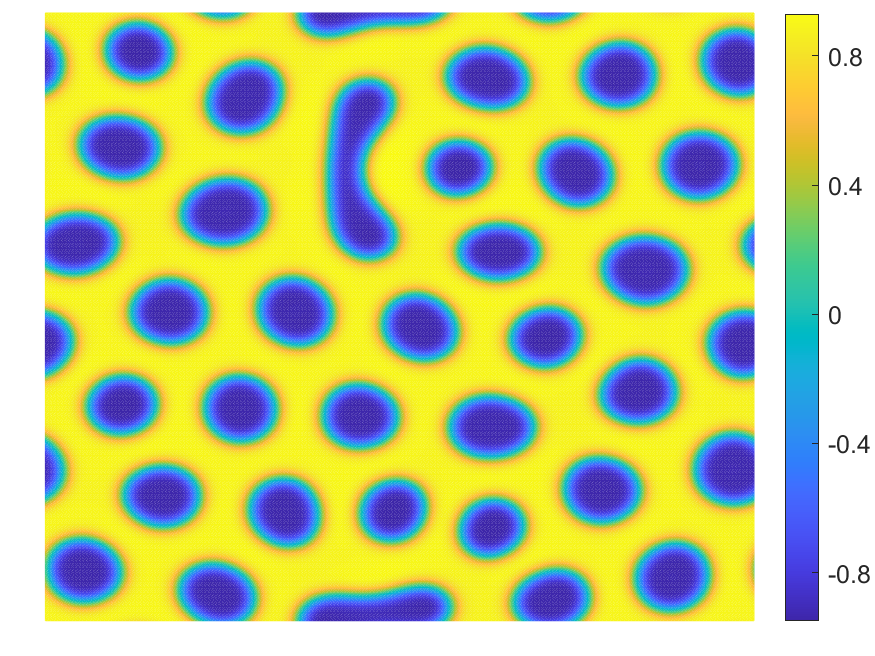}}
\hspace{-3mm}
{\includegraphics[height=5cm,width=5.49cm]{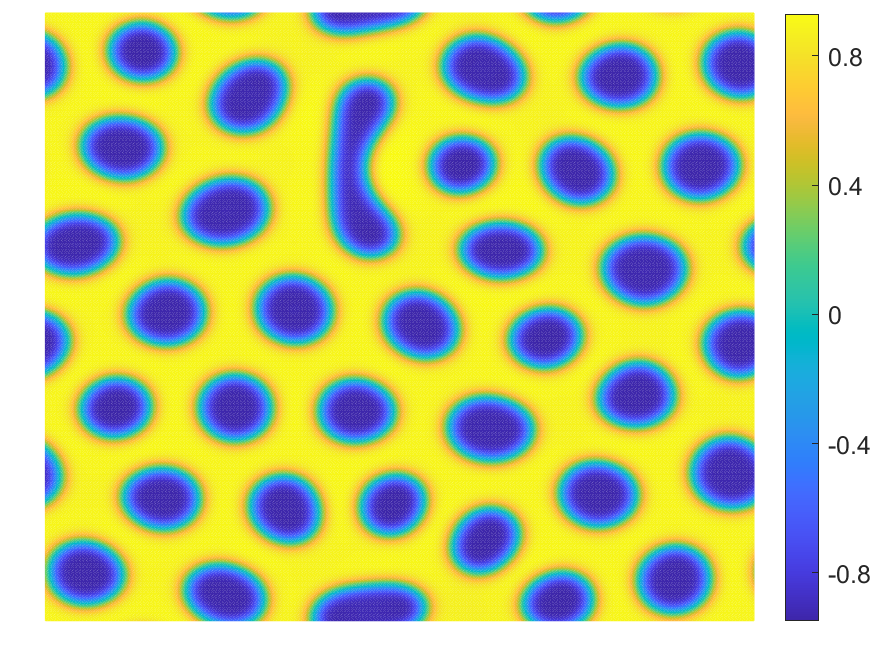}}
\hspace{-3mm}
{\includegraphics[height=5cm,width=5.49cm]{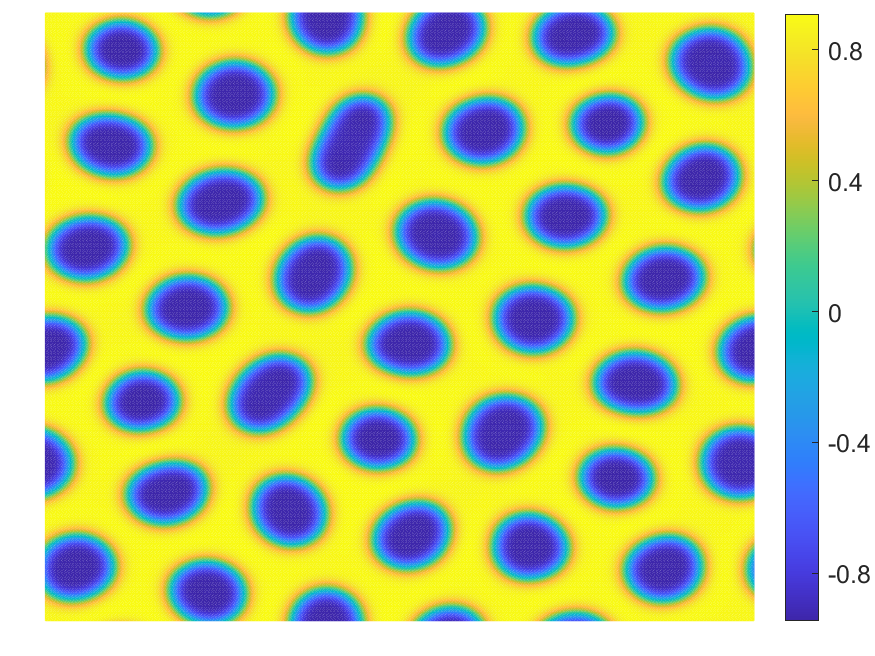}}
\caption{The phase structures at $t=10$, $20$, $50$, $100$, $200$, and $2000$, respectively (left to right and top to bottom) generated by the P-ETDRK2 scheme with $\kappa=2$, $\tau=0.1$, $h=1/512$, $\epsilon=0.02$, $\theta_c=1.6$, $\theta=0.8$, $\sigma=70$ }
\label{phase_structure}
\end{figure}

\begin{figure}[h!]
\centering
{\includegraphics[width=8.1cm]{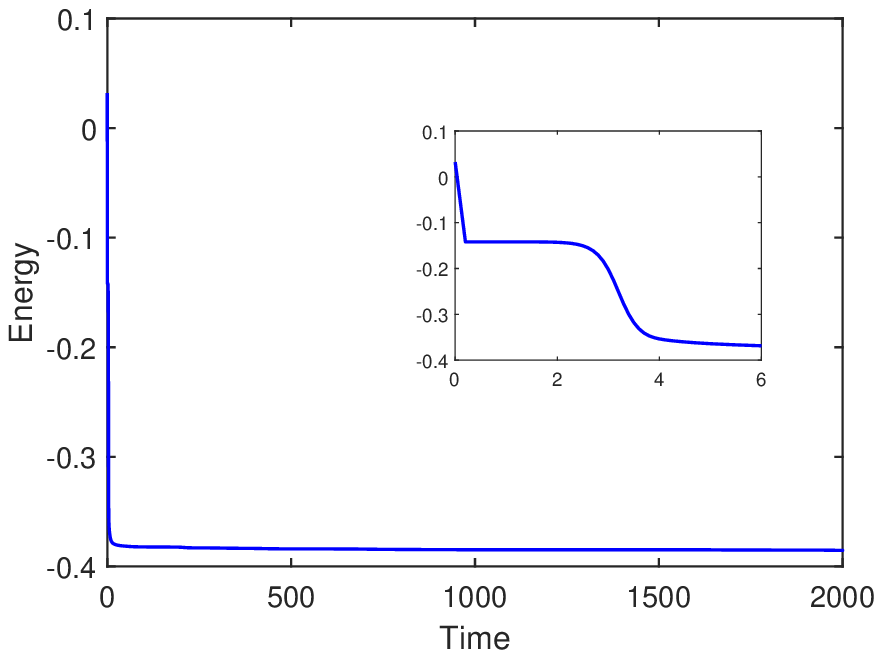}}
{\includegraphics[width=8.1cm]{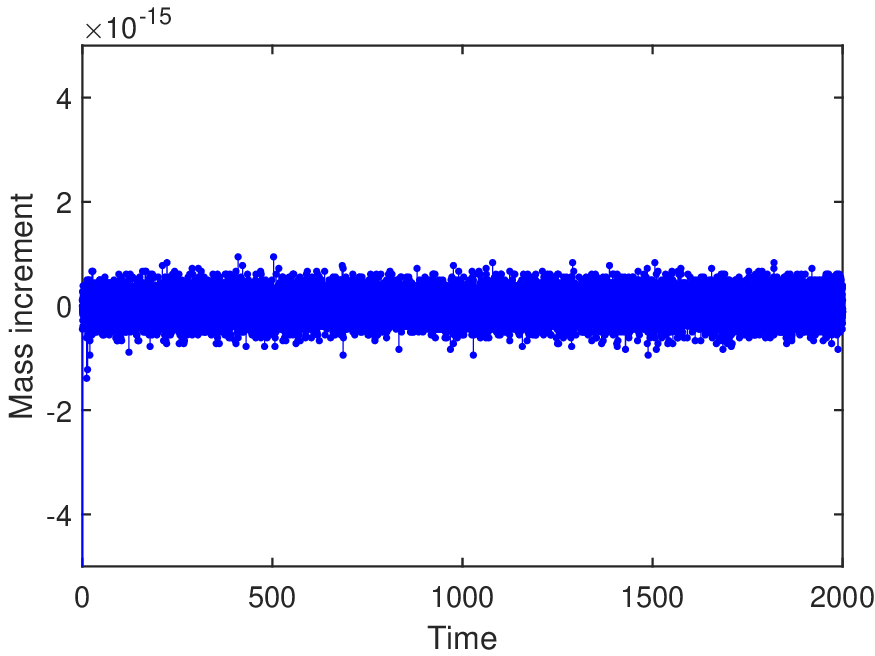}}
\caption{The evolutions of the energy and the mass increment of the numerical solution generated by the P-ETDRK2  scheme with $\kappa=2$, $\tau=0.1$, $h=1/512$, $\epsilon=0.0 2$, $\theta_c=1.6$, $\theta=0.8$, $\sigma=70$ }
\label{energy_mass}
\end{figure}

\subsection{The effect of the nonlocal constant $\sigma$}

The nonlocal constant $\sigma$ represents the long-range interactions, which is inversely proportional to the index of polymerization of the block copolymers. This implies the monomers are unlikely to aggregate, under the condition that $\sigma$ is relatively large. Consequently, as the nonlocal parameter $\sigma$ increases, a greater number of spherical structures emerge in the sphere phase in the domain. 

Now we investigate the effect of the nonlocal constant $\sigma$ on the phase structure. The nonlocal constant is chosen as $\sigma=1$, $5$, $10$, $20$, $30$, $40$, $60$, $70$, other parameters and the initial value is the same as the parameter chosen in Subsection \ref{coarsening_dynamics}. From phase structures depicted in Figure \ref{effect_sigma}, we observe that the number of the spherical structures $N_{\sigma}$ change as $N_\sigma=5$, $11$, $18$, $29$, $35$, $44$, $57$, $62$ when $\sigma=1$, $5$, $10$, $20$, $30$, $40$, $60$, $70$, respectively.  More importantly, the relation between the number of the spherical structures $N_\sigma$ and the nonlocal constant $\sigma$ can be approximately described as $N_\sigma \sim O(\sigma^{2/3})$, which is consistent with the theoretical analysis in \cite{OK1986}.

\begin{figure}[h!]\label{effect_sigma}
\hspace{-25mm}
{\includegraphics[scale=0.85
]{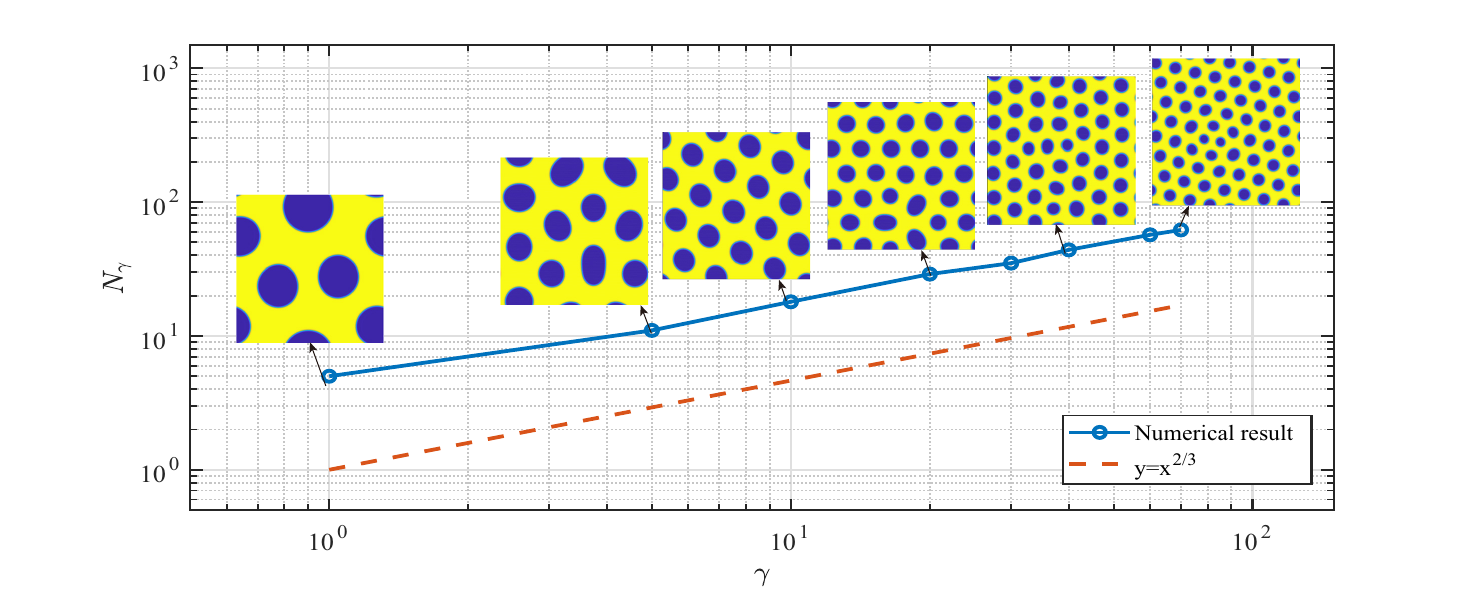}}
\hspace{-30mm}
\caption{The log-log plot of dependence the number of spherical structure on $\sigma$ with $\kappa=2$, $\tau=0.1$, $h=1/512$, $\epsilon=0.02$, $\theta_c=1.6$, $\theta=0.8$ and $\sigma=1$, $5$, $10$, $20$, $30$, $40$, $60$, $70$ }
\end{figure}

\section{Conclusion}\label{conclusion}

In this work, we design the efficient numerical schemes for solving the nonlocal Cahn-Hilliard equation with Flory--Huggins potential by using first- and second-order P-ETD method in time and central finite difference in space. In addition, the fully-discrete convergence analysis for both first- and second-order schemes is presented. Furthermore, the theoretical findings is verified by the numerical experiments in different cases especially in the aspects of  preserving bound and mass conservation. Finally, we investigate the effect of the nonlocal parameter on the phase structure.

\end{document}